\documentclass[10pt]{article}
\usepackage{amsmath,amssymb,amsthm,amsfonts,amstext,amsbsy,amscd}
\usepackage{mathrsfs}
\usepackage{natbib}

\usepackage{bm}
\usepackage{multirow}
\usepackage{multicol}
\usepackage{enumerate}
\usepackage{latexsym}
\usepackage[top=3cm, bottom=3.5cm, left=3cm, right=3cm]{geometry}

\usepackage{dsfont}
\usepackage{color}
\usepackage{tikz}
\RequirePackage[colorlinks,citecolor=blue,urlcolor=blue]{hyperref}
\RequirePackage{hypernat}

\definecolor{C}{RGB}{9,150,80}

\usepackage{hyperref}
\newcommand{\PP}{\mathbb{P}}

\newcommand{\E}{\mathbb{E}}

\newcommand{\1}{\mathbf{1}}
\newcommand{\R}{\mathbb{R}}

\newcommand{\eps}{\varepsilon}

\newtheorem{theorem}{Theorem}

\newtheorem{lemma}{Lemma}

\newtheorem{corollary}{Corollary}

\renewcommand{\tilde}{\widetilde}
\renewcommand{\hat}{\widehat}

\begin{document}
\title{Non-asymptotic control of the cumulative distribution function of L\'evy processes}
\author{C\'eline Duval  \footnote{Universit\'e de Paris,  MAP5, UMR CNRS
8145, France. E-mail: \href{celine.duval@parisdescartes.fr}{celine.duval@parisdescartes.fr}.} \ \ and \ \ Ester Mariucci \footnote{Otto von Guericke Universität Magdeburg, Germany.
    E-mail: \href{mailto: ester.mariucci@ovgu.de}{ester.mariucci@ovgu.de}.}}
\date{}

\maketitle

\begin{abstract} 
We propose non-asymptotic controls of the cumulative distribution function $\PP(|X_{t}|\ge \eps)$, for any $t>0$, $\eps>0$ and any L\'evy process $X$ such that its L\'evy density is bounded from above by the density of an $\alpha$-stable type L\'evy process in a neighborhood of the origin. The results presented are non-asymptotic and optimal, they apply to a large class of L\'evy processes.
\end{abstract}

\section{Introduction and motivations}

The law of any L\'evy process $X$ is the convolution between a Gaussian process, the martingale  $M$ describing its small jumps and a compound Poisson process. However, for most L\'evy processes a closed form expression for the law of their increments is not  known. The core of the problem lies in computing the distribution of the small jumps. This technical limitation makes both inference and simulations difficult for L\'evy processes. To cope with this shortcoming it is usual to approximate a general L\'evy process $X$ with a family of compound Poisson processes by ignoring the jumps smaller than some level $\varepsilon$. Also, when the L\'evy measure is of infinite variation, solutions that consist in approximating the law of $M_t$ with a Gaussian distribution are motivated by  \cite{gnedenko1954limit} (see also \cite{tankov}, \cite{cohen2007gaussian} or \cite{carpentier2018total}). This type of approximations are of interest because both Gaussian and compound Poisson processes are nowadays well understood, both in terms of continuous and discrete observations. The same cannot be said for the small jumps which remain  complex objects, difficult
to manipulate.

 In order to quantify the precision of such approximations it becomes of crucial importance to have a sharp control of quantities such as $\mathbb P(|X_t|>\varepsilon)$ and $\PP(|M_t|>\varepsilon)$. 
The issue, besides being interesting in itself, are sometimes required, for example to study non-asymptotic risk bounds for estimators of the L\'evy density  from discrete observations of $X$ (see \cite{LH} or \cite{duval2017compound}). This has important consequences in various fields of application where L\'evy processes are commonly used to describe real life phenomena. The literature on the applications of L\'evy processes is extensive, ranging from financial, biology, geophysics and neuroscience, to name but a few. In this respect, we will limit ourselves to mention \cite{barndorff2012levy} and the references therein. \\

 Formally, a L\'evy process $X$ is characterized  by its L\'evy triplet $(b,\Sigma^{2},\nu)$ where $b\in\R$, $\Sigma\ge 0$ and $\nu$ is a Borel measure on $\R$ such that
\begin{align*}
\nu(\{0\})=0 \quad \textnormal{ and }\quad  \int_{\R}(y^2\wedge 1)\nu(d y)<\infty.
\end{align*} 
The L\'evy-It\^o decomposition (see \cite{Bertoin}) allows to write a L\'evy process $X$ of L\'evy triplet $(b,\Sigma^{2},\nu)$ as the sum of four independent L\'evy processes, for all $t\ge0$, 
\begin{align}\label{eq:Xito1}
 X_t&=tb+\Sigma W_{t}+\lim_{\eta\to 0}\bigg(\sum_{s\leq t}\Delta X_s\1_{(\eta,1]}(|\Delta X_s|)-t\int_{\eta<|x|\leq 1}x\nu(d x)\bigg) + \sum_{s\leq t}\Delta X_s\1_{(1,\infty)}(|\Delta X_s|)\nonumber \\&=: tb+\Sigma W_{t}+M_t+Z_t,
\end{align}
where  $\Delta X_r$ denotes the jump at time $r$ of the càdlàg process $X$: $\Delta X_r=X_r-\lim_{s\uparrow r} X_s. $  The first term is a deterministic drift, $ W$ is a standard Brownian motion which is path-wise continuous and $M$ and $Z$ compose the discontinuous jump part of $X$. 
The process  $M$ is a centered martingale gathering the {small jumps} \textit{i.e.} the jumps of size smaller than $1$ and it has L\'evy measure $\1_{|x|\leq 1}\nu$.    The process 
 $Z$ instead, is a compound Poisson process gathering jumps larger than 1 in absolute value, it has L\'evy measure $\1_{|x|> 1}\nu$. 
In the sequel we make $(b,\Sigma)=(\gamma_\nu,0)$ with
\begin{equation*}
\gamma_\nu:=\begin{cases}
\int_{|x|\leq1 }x\nu(dx) & \mbox{if }\int_{|x|\leq 1}|x|\nu(dx)<\infty \\
0 & \mbox{if } \int_{|x|\leq 1}|x|\nu(dx)=\infty,
 \end{cases}
\end{equation*}
 (see Section \ref{sec:ext} for a discussion in the general case)  and rewrite \eqref{eq:Xito1}  as
 \begin{equation}\label{eq:x}
X_t=t b(\eps) + M_t(\eps)+ Z_t(\eps),\quad \forall 1\ge \eps>0, 
\end{equation} where, 
\begin{equation*}
b(\varepsilon):=\begin{cases}\,\ \ 
             \int_{|x|\leq \varepsilon } x\nu(dx)\quad &\text{if}\quad \int_{|x|\leq 1}|x|\nu(dx)<\infty\\
             -\int_{\varepsilon\leq |x|\leq 1} x\nu(dx)\quad &\text{if}\quad \int_{|x|\leq 1}|x|\nu(dx)=\infty ,
             \end{cases}   
\end{equation*}
$M(\eps)=(M_t(\eps))_{t\geq 0}$ is a L\'evy process accounting for the centered jumps of $X$ with size smaller than $\eps$:
$$M_t(\eps)=\lim_{\eta\to 0} \bigg(\sum_{s\leq t} \Delta X_s\1_{\eta<|\Delta X_s|\leq \eps}-t\int_{\eta<|x|\leq \eps} x\nu(dx)\bigg),$$
and  $Z(\eps)=(Z_t(\eps))_{t\geq 0}$ is a compound Poisson process of the form 
$
Z_t(\eps):=\sum_{i=1}^{N_t(\eps)} Y_i^{(\eps)},
$
where $N(\eps)=(N_t(\eps))_{t\geq 0}$ is a Poisson process of intensity $\lambda_\eps:=\int_{|x|> \eps} \nu(dx)$ independent of the sequence of i.i.d. random variables $(Y_i^{(\eps)})_{i\geq 1}$ with common law $\nu_{|_{\R\setminus[-\eps,\eps]}}/\lambda_\eps$. In the sequel we use the notations $a\wedge b=\min(a,b)$ and $a\vee b=\max(a,b)$.\\

A first well known result (see \emph{e.g.} \cite{Bertoin} Section I.5 or Corollary 3 in \cite{Rusch2002}) relates the L\'evy measure to the limit of $\PP(|X_{t}|\ge \eps)$ as $t\to 0$ as follows. 
\begin{lemma}\label{lemma0} Let $X$ be a L\'evy process with L\'evy measure $\nu$. For all $\eps>0$ it holds that
\begin{equation*}
\lim_{t\to0}\frac{ \PP(|X_t|\geq \eps)}{t}=\int_{\R\setminus[-\eps,\eps]} \nu(dy).\end{equation*}
In particular, it leads to
\begin{equation*}
\lim_{t\to0}\frac{\PP(|M_t(\eps)|\geq\eps)}{t}=0 \quad\text{ and }\quad \lim_{t\to0}\frac{\PP(|tb(\eps)+M_t(\eps)|\geq\eps)}{t}=0.
\end{equation*}
\end{lemma}
Lemma \ref{lemma0} suggests that $\PP(|X_{t}|\ge \eps)\asymp \lambda_{\eps}t$ ``for $t$ small enough'', however it gives no information on how small $t$ should be, nor on the size of the error term $\PP(|X_{t}|\ge \eps)- \lambda_{\eps}t$ nor on what happens if $\eps$ gets small.
Of course, $ \PP(|X_t|\geq \eps)$ and $\PP(|M_t(\eps)|\geq\eps)$ can  be controlled with elementary inequalities, such as the Markov inequality, but this often leads to sub-optimal results. Indeed, the Markov inequality gives $\PP(M_t(\eps)>\eps)\leq t\sigma^2(\eps)\eps^{-2},$ if we denote by $\sigma^2(\eps):=\int_{-\eps}^\eps x^2\nu(dx),$ the variance of $M_1(\eps)$, whereas a sharper result, can be achieved using the Chernov inequality as follows.
\begin{lemma}\label{sj}
For any $\eps\in(0,1]$, $t>0$ and $x>0$, it holds:
\begin{align*}
\PP(|M_t(\eps)|>x)&\leq 2 e^{\frac{x}{\eps}}\bigg(\frac{t\sigma^2(\eps)}{x\eps+t\sigma^2(\eps)}\bigg)^{\frac{x\eps+t\sigma^2(\eps)}{\eps^2}}.
\end{align*}
Moreover, if $t\sigma^2(\eps)\eps^{-2}\leq 1$, it leads to 
\begin{align}
\label{sg}\PP(M_t(\eps)>x)&\leq \bigg(\frac{e\sigma^2(\eps)}{\eps^2}\bigg)^{\frac{x}{\eps}} e^{e^{-1}}t^{\frac{x}{\eps}} \ \text{ and } \ \PP(M_t(\eps)\leq -x)\leq \bigg(\frac{e\sigma^2(\eps)}{\eps^2}\bigg)^{\frac{x}{\eps}} e^{e^{-1}}t^{\frac{x}{\eps}}.
\end{align}

\end{lemma}
 Lemma \ref{sj} is a  modification of Remark 3.1 in \cite{LH}. A similar result can also be obtained using martingale arguments (see \cite{vanzanten}, Theorem 4.1). 
Again Lemma \ref{sj} is suboptimal as it does  not allow to derive that $\lim_{t\to0}{\PP(M_t(\eps)\geq\eps)}/{t}=0.$ If we want to be more precise about the behavior for $t\to 0$ we need additional assumptions.\\

 Studying the behavior in small times of the transition density of L\'evy process goes back to \cite{leandre1987densite} (see also \cite{ishikawa1994asymptotic}) and is carried in the real case in  \cite{picard1997density} which is also interested in the behavior of the supremum of this quantity and its derivatives. For the cumulative distribution function, expansions of order 2 for $\PP(X_{t}\ge y),$ for fixed $y$ and $t$ going to 0, are given in \cite{marchal2009small} in the particular cases where $X$ is the sum of a compound Poisson process and either a Brownian motion or an $\alpha$-stable process.  
 
 The most complete results can be found in  \cite{LH}, which, for general L\'evy processes, establishes asymptotic expansions  at any order of $\PP(X_{t}\ge y)$, for fixed $y$ bounded away from 0 and $t\to0$.  
They prove that 
 \[
 \PP(X_{t}\ge y)=e^{-\lambda_{\eta}t}\sum_{j=1}^{n}c_{j}t^{j}+O_{\eta,\underline{y}}(t^{n+1}),
 \]
where $n\ge 1$, $0<\eta<1\wedge\big( \underline{y}/(n+1)\big)$, the L\'evy density has $2n+1$ bounded derivatives away from the origin, $y\ge \underline{y}$ and $0<t<t_{0}$, for some $\underline{y}$ and $t_0$. No bounds on either $\underline{y}$ or $t_0$ are provided.

In the case $n=1$, they further prove that
\[
d_2(y) = \lim_{t\to0} \frac1t\Big(\frac1t\PP(X_{t}>y)-\nu((y,\infty))\Big)
\]
exists, when the L\'evy density $f$ is bounded outside the interval $[-\eta,\eta]$, $0<\eta<y/2\wedge 1$, and either $f$ is $C^{1}$ in a neighborhood of $y$, or $f$ is continuous in a neighborhood of $y$, of bounded variation and $\Sigma=0$ (defined as in \eqref{eq:Xito1}). This is again an asymptotic result; therefore, it provides no information on how small $t$ should be for the approximation of $\PP(X_{t}>y)-t\nu((y,\infty))$ by $d_{2}(y)t^{2}$ to be accurate. Moreover, even though they give an explicit characterization of $d_2(y)$, this does not translate in a readily understandable dependency on $y$.

Our main contribution is a non-asymptotic control of $\PP(|X_{t}|\ge \eps)$, which is valid for any $\eps>0$ and any $0<t<t_{0}(\eps)$. A lot of effort has been made to make the dependency on $\eps$ explicit, both in $t_{0}(\eps)$ and in the final bound. Concerning the hypotheses on the L\'evy density $f$, in the finite variation case we do not require any continuity, but only that it is bounded from above by an $\alpha$-stable like density in a neighborhood of 0, see the definition on the class $\mathscr L_{M,\alpha}$ below. In the setting of infinite variation, we distinguish two cases: when $f$ is also Lipschitz continuous in a neighborhood of $\eps$ (a similar condition to that of \cite{LH}), we find a non-asymptotic bound of the order of $t^2$. We also analyze the case where the continuity hypothesis on $f$ is dropped. Then, the order in $t$ of the non-asymptotic bound deteriorates to $t^{1 + 1/\alpha}$, $1 \leq \alpha < 2$. This is not an artifact of the proof, as an example in \cite{marchal2009small} indicates.

The case of the small jumps is treated separately as an intermediate step to the general case (see Theorems \ref{ps1} and \ref{ps2}). We think that these results are of independent interest and provide a new insight on the process of the small jumps. Finally, our proofs are elementary and self-contained and they do not rely on the use of the infinitesimal generator.\\

Next Section \ref{Main} gathers the main results of the paper. We begin with defining the classes of L\'evy densities that we consider. On these classes we provide a non-asymptotic control of $\PP(|M_{t}(\eps)|\ge \eps)$ and  $\PP(|X_{t}|\ge \eps)$. We consider separately finite variation L\'evy processes and infinite variation L\'evy processes, for which we only detail the symmetric case. In both cases our results permit to recover Lemma \ref{lemma0}. We compare our results to examples for which the quantity  $\PP(|X_{t}|\ge \eps)$ is known. Section \ref{Main} ends with a discussion on the validity of the results in presence of a Brownian component.  Section \ref{sec:prf} gathers the proofs of the main results whereas in Appendix \ref{secApp} all auxiliary results are established and the computations of the examples are carried out.

\section{Non-asymptotic expansions}\label{Main}

Consider $\alpha\in(0,2)$ and $M$ be positive constants, define the classes of fonctions
\begin{align*}
\mathscr L_{M,\alpha}&:=\bigg\{f : f(x)\leq \frac{M}{|x|^{1+\alpha}},\quad  \forall |x|\leq 2\bigg\},\hspace{1cm}
\mathscr L_{M}:=\bigg\{f : \sup_{|x|\geq 1}|f(x)|\leq M\bigg\}.
\end{align*}
A L\'evy density $f$ belongs to the class $\mathscr L_{M}$, $M>0$, if it is bounded outside a neighborhood of the origin. It belongs to  $\mathscr L_{M,\alpha}$, $M>0$ and $\alpha>0$, if $\sup_{x\in[-2,2]}f(x)|x|^{1+\alpha}\le M$. In particular $\mathscr L_{M,\alpha}$ contains any $\tilde\alpha$-stable L\'evy density such that $\tilde\alpha\le\alpha$. Also any finite variation L\'evy process is in the class $\mathscr	L_{M,1},$ for some positive $M$. We stress that no lower bound condition is required for the L\'evy density.

\subsection{Finite variation L\'evy processes}

We state two non-asymptotic results offering a control of the distribution function of a finite variation L\'evy process.

\begin{theorem}\label{ps1}
Let $\nu$ be a L\'evy measure absolutely continuous with respect to the Lebesgue measure and denote by $f=\frac{d\nu}{dx}$. Let $\eps\in(0,1]$, $\alpha\in(0,1)$, $M>0$ and $f\in\mathscr L_{M,\alpha}$. Then, there exists a constant $    {\rm C}_1 >0$, only depending on $\alpha$, such that
$$\PP(|tb(\eps)+M_t(\eps)|\geq \eps)\leq 2t^2 M^2  {\rm C}_1 \eps^{-2\alpha},\quad \forall \ 0<t\leq \frac{(1-\alpha)\eps^{\alpha}}{M4^{1+\alpha}} .$$
If, in addition, $f$ is a symmetric function, then there exists a constant $    {\rm C}_2 >0$, only depending on $\alpha$, such that
$$\PP(|M_t(\eps)|\geq \eps)\leq 2t^2 M^2   {\rm C}_2  \eps^{-2\alpha},\quad \forall \ 0<t\leq \frac{\eps^\alpha(2-\alpha)}{M2^{\alpha+1}} .$$
Explicit formulas for the constants $  {\rm C}_1$ and $  {\rm C}_2$ are given in \eqref{ctealpha} and \eqref{sjsym}, respectively.
\end{theorem}
Theorem \ref{ps1} highlights how likely  the process of the jumps smaller than $\eps$ are to present excursions larger than their size $\eps$ in a time interval of length $t$. When dealing with a discretized trajectory of a L\'evy process, this provides relevant information on the contribution of the small jumps to the value of the observed increment. The following result generalizes Theorem \ref{ps1} to any L\'evy process with a L\'evy density in $\mathscr L_{M,\alpha}$, $\alpha\in(0,1)$ or in $\mathscr L_{M,\alpha}\cap \mathscr L_{M}$ if $\eps>1$. In particular it permits to derive an order of the rate of convergence in Lemma \ref{lemma0}.

\begin{theorem}\label{teo1}
Let $X_t=\sum_{s\leq t} \Delta X_s$ be a finite variation L\'evy process with L\'evy measure $\nu$ absolutely continuous with respect to the Lebesgue measure and denote by $f=\frac{d\nu}{dx}$. 

\begin{itemize}
\item If $\eps\in(0,1]$ and $f\in\mathscr L_{M,\alpha}$ for some $\alpha\in(0,1)$ and $M>0$, then for all $0<t< (1-\alpha) M^{-1} \eps^{\alpha}4^{-(1+\alpha)}$ it holds
\begin{equation*}
|\PP(|X_t|>\eps)-\lambda_\eps t|\leq t^2M^2\eps^{-2\alpha}\big( 2{\rm C_1}+ {\rm D}_1\big)+t^2M\lambda_\eps\eps^{-\alpha} {\rm D}_2+2t^2\lambda_\eps^2,
\end{equation*}
where $\rm C_1$, $\rm D_1$ and $\rm D_2$  only depend on $\alpha$ and are defined in \eqref{ctealpha} and \eqref{cte3}.
\item If $\eps>1$ and $f\in\mathscr L_{M,\alpha}\cap \mathscr L_M$ for some $\alpha\in(0,1)$ and $M>0$, then for all $0<t<(1-\alpha)(5M)^{-1}$ it holds
\begin{align*}
|\PP(|X_t|>\eps)-\lambda_\eps t|&\leq 2M^2t^2 (\tilde{\rm D}_1+ {\rm C_1})+2t^2\lambda_1^2+2Mt^2\bigg(\frac{4}{2-\alpha}(\eps-3/2-t|b(1)|)\1_{\eps>3/2+t|b(1)}
\Big)\\
&\quad +Mt^2\bigg(4\times 5^{\alpha}\1_{1<\eps<1+2t|b(1)|}+\frac{8}{5}+\frac32\lambda_2+\frac{4\lambda_1}{2-\alpha}\bigg),
\end{align*}
where $\rm C_1$ and $\tilde{\rm D}_1$ only depend on $\alpha$ and  are defined in \eqref{ctealpha} and \eqref{tdlm3}.
\end{itemize}
If in addition we suppose that $\nu$ is a symmetric measure, then 
\begin{itemize}
\item If $\eps\in(0,1]$ and $f\in\mathscr L_{M,\alpha}$ for some $\alpha\in(0,1)$ and $M>0$, for any $0<t<\eps^\alpha(2-\alpha)M^{-1}2^{-\alpha-1}$   it holds
\begin{equation*}
|\PP(|X_t|>\eps)-\lambda_\eps t|\leq  2t^2M^2\eps^{-2\alpha}({\rm C_2}+{\rm D_3})+\frac{t^2M}{2(2-\alpha)}\big(\lambda_\eps \eps^{-\alpha}+4\lambda_{2\eps} \eps^{-\alpha}\big)+2t^2\lambda_\eps^2,
\end{equation*}
where $\rm C_2$ and $\rm D_3$ only depend on $\alpha$ and are defined in \eqref{sjsym} and \eqref{d2lm3}.

\item If $\eps>1$ and $f\in\mathscr L_{M,\alpha}\cap \mathscr L_M$ for some $\alpha\in(0,1)$ and $M>0$, for any $0<t<(2-\alpha)M^{-1}2^{-\alpha-1}$   it holds
\begin{equation*}
|\PP(|X_t|>\eps)-\lambda_\eps t|\leq  2t^2M^2{\rm C_2}+\frac{t^2M}{2-\alpha}\bigg(\lambda_1 2^{-\alpha}+\frac{4M}{\alpha(1-\alpha)}+\lambda_{1+\eps} \bigg)+2t^2\lambda_1^2,
\end{equation*}
where $\rm C_2$ is defined in \eqref{sjsym}.
\end{itemize}

\end{theorem}

 The results of Theorems \ref{ps1} and \ref{teo1} are non-asymptotic.  If we apply Theorem \ref{teo1} to a L\'evy process $X$ whose L\'evy measure $\nu$ is concentrated on $[-\eps,\eps]$, for $\eps\in(0,1]$, we recover the result of Theorem \ref{ps1} up to the constant ${\rm D}_{1}$  as in that  case $\lambda_{\eps}=0$. Though, Theorem \ref{ps1} is not a corollary of Theorem \ref{teo1} as the proof of the latter results uses Theorem \ref{ps1}. 
 
These results show that for a finite variation L\'evy process whose L\'evy density lies in $\mathscr L_{M,\alpha}$, for some $\alpha\in(0,1)$ and $M>0$, the discrepancy between $\PP(|X_{t}|>\eps)$ and $\lambda_{\eps}t$ is in $t^{2}$. Moreover, as the role of the cutoff $\eps$ is made explicit in the upper bound, it is possible to measure the accuracy of this approximation when $\eps$ gets small. Then, the rate of the upper bound is --up to a constant-- in  $t^{2}(\eps^{-2\alpha}\vee\lambda_{\eps}\eps^{-\alpha}\vee \lambda_{\eps}^{2})$. For example for  an $\alpha$-stable process with $\alpha\in(0,1)$ this order simplifies in $t^{2}\lambda_{\eps}^{2}.$

\subsection{Symmetric infinite variation L\'evy processes}
We generalize Theorems \ref{ps1} and \ref{teo1} to symmetric infinite variation L\'evy processes whose L\'evy density lies in $\mathscr L_{M,\alpha},\ \alpha\in[1,2)$ and $M>0$. 
\begin{theorem}\label{ps2}
Let $\nu$ be a symmetric L\'evy measure absolutely continuous with respect to the Lebesgue measure and denote by $f=\frac{d\nu}{dx}$. Let $\eps\in(0,1]$, $\alpha\in[1,2)$, $f\in\mathscr L_{M,\alpha}$ and $0<t<(\eps/2)^\alpha(1\land ((2-\alpha)/2M))$. Then, there exists a constant $ {\rm E}_{1}>0$, only depending on $\alpha$ (see \eqref{cteT3}), such that
\begin{align*}
\PP(|M_t(\eps)|\geq \eps)&\leq\frac{2^{2+\alpha}M t^{1+1/\alpha}}{\eps^{1+\alpha}}\bigg(1+\frac{M}{\alpha(2-\alpha)(\alpha-1)}\bigg)+ 2t^2 M^2 {\rm E}_{1} \eps^{-2\alpha}, \quad \alpha\in(1,2),\\
\PP(M_t(\eps)\geq \eps)&\leq  \frac{4t^2M^2}{\eps^2}\bigg(e^{2+1/e}+ \frac{37}{9}\bigg)+\frac{4M t^2}{\eps^2}+\frac{16M^2}{\eps^2}t^2 \ln\Big(\frac{\eps}{2t}\Big), \quad \alpha=1.
\end{align*}
\end{theorem}

\begin{theorem}\label{lambda2bis}
Let $\nu$ be a symmetric L\'evy measure with density $f$ with respect to the Lebesgue measure and $f\in \mathscr L_{M,\alpha}\cap \mathscr L_M$ for some $\alpha\in[1,2)$ and $M>0$. Then, for all $0<t<(\eps\wedge1/2)^\alpha\big(1\land ((2-\alpha)/2M)\big)$, $\eps>0$, it holds: 
\begin{align*}
|\PP(|X_t|>\eps)-&\lambda_\eps t|\leq {\rm G}_{1}\frac{t^{1+1/\alpha}}{(\eps\wedge 1)^{1+\alpha}}+{\rm G}_{2}\frac{t^{2}}{(\eps\wedge 1)^{2\alpha}}+\frac{5M}{2-\alpha}\frac{t^{2}\lambda_{1}}{(\eps\wedge1)^2}+\frac{4M^{2}t^{2}\eps}{2-\alpha}\1_{\eps>2} \\
&\quad+M^2t^{2}\1_{\alpha=1}\bigg(\frac{12}{\eps\wedge 1}\ln\Big(\frac{C(1\wedge\eps\wedge(\eps-1\vee 0))}{t}\Big)+\frac{16}{\eps^2} \ln\Big(\frac{\eps}{2t}\Big)\bigg)+2\lambda_{\eps\wedge1}^{2}t^{2},\end{align*}
where $C:=\big(1\wedge\big((2-\alpha)/2M\big)\big)^{1/\alpha}$ and ${\rm G}_{1}$  and ${\rm G}_{2}$ are positive constants, only depending on $M$ and $\alpha$, defined in \eqref{eq:G1G2}.

\end{theorem}

Compared to Theorems \ref{ps1} and \ref{teo1}  the rates of Theorems \ref{ps2} and \ref{lambda2bis} are slower as $t^{2}\le t^{1+1/\alpha}$ for $\alpha\in(1,2)$. Nevertheless, the rate $t^{1+1/\alpha}$ of Theorems \ref{ps2} and \ref{lambda2bis} seems optimal. Indeed,  as shown in Remark 3.5 of \cite{LH} (see also \cite{marchal2009small}) it is possible to build a discontinuous L\'evy measure $f$ as the sum of an $\alpha$-stable L\'evy process plus a compound Poisson process presenting a discontinuity at $\eps$ that lies in $\mathscr L_{M,\alpha}$ and attains this rate $t^{1+1/\alpha}$. Adding a regularity assumption on $f$ on a neighborhood of $\eps$, it is possible to have a finer bound in $t^{2}$ as established in the following result.

\begin{theorem}\label{lambda2}
Let $\nu$ be a symmetric L\'evy measure having a density $f$ with respect to the Lebesgue measure with $f\in \mathscr L_{M,\alpha}$ for some $\alpha\in[1,2)$ and $M>0$. Let $\eps>0$ and assume that $f$ is  $M(\eps\wedge 1)^{-(2+\alpha)}$-Lipschitz on the interval $((3/4(\eps\wedge1),2\eps-3/4(\eps\wedge1))$. For all $0<t\leq  \frac{(2-\alpha)(1\wedge\eps)^\alpha}{2^{1+\alpha}M}$, it holds:
\begin{align*}
|\PP(|X_t|>\eps)-\lambda_\eps t|&\leq
 t^2M^2 \big( ({\rm F}_{1}\eps^{-2\alpha}+\lambda_1\eps^{-\alpha}{\rm F}_2) \1_{0<\eps\leq 1}+(\eps^2{\rm F}_3 +{\rm F}_{4})\1_{\eps>1} \big) +2t^2\lambda_1^2 +\frac{t^4 M^4 {\rm F}_5}{(\eps\wedge 1)^{4\alpha}},
\end{align*}
where ${\rm F}_1,\dots, {\rm F}_5$ are universal positive constants, only depending on $\alpha$, defined in \eqref{d1d2}.
\end{theorem}

First, note that any L\'evy density $f$ that writes as $L(x)/x^{1+\alpha}$ for $x\in[-2,2]\setminus \{0\}$, where $L$ is differentiable, bounded, with bounded derivative and $\alpha\in[1,2)$ satisfies the assumptions of Theorem \ref{lambda2}. Moreover, under the latter assumption, Theorem \ref{lambda2} applied to a L\'evy process $X$ whose L\'evy density $f$ is concentrated on $[-\eps,\eps]$, $\eps\in(0,1],$ leads to a finer rate than the one of Theorem \ref{ps2}, namely,
\begin{align*}
\PP(|M_t(\eps)|>\eps)&\leq
 t^2M^2 ({\rm F}_{1}\eps^{-2\alpha}+\lambda_1\eps^{-\alpha}{\rm F}_2) +2t^2\lambda_1^2+\frac{t^4 M^4 {\rm F}_5}{(\eps\wedge 1)^{4\alpha}}.
\end{align*}

\subsection{Discussion}

The results of Theorems \ref{ps1} to \ref{lambda2} are non-asymptotic and show the impact of the cutoff $\eps$ in the constants. In particular they permit to recover, for every fixed $\eps>0$, on the classes considered, the result of Lemma \ref{lemma0} having $t\to0$. \medskip

\noindent\textbf{Optimality of the results} The rates of Theorems  \ref{ps1}, \ref{teo1} and \ref{lambda2}  are of the form $t^{2}(\eps\wedge 1)^{-2\alpha}$, up to a constant depending on $M$ and $\alpha$. This quantity is optimal in $t$ on the considered classes. Indeed, in next Section \ref{sec:ex} we show that for compound Poisson processes, for which explicit calculations can be performed and which are included in $\mathscr	L_{M,\alpha}$ for all $\alpha\in(0,2),$ examples can be built attaining this rate. As already highlighted, the rate of Theorem \ref{ps2} is also optimal. The dependency  in $\eps$  of the constant $\eps^{-2\alpha}$ also appears to be the right one, since, for an $\alpha$-stable process, it holds that $\lambda_{\eps}=O(\eps^{-\alpha})$.
Therefore, in general it is not possible to improve the rates derived in these Theorems, even though this might be possible on specific examples (see the Cauchy process in Section \ref{sec:ex}).\medskip

\noindent\textbf{Strategy of the proofs} All the proofs are self-contained, they rely on the decomposition  \eqref{eq:x}, which holds for any L\'evy process and any level $\eps>0$, and Lemma \ref{sj}. More precisely, to establish Theorems \ref{teo1} and \ref{lambda2bis}, we consider the decomposition   \eqref{eq:x}, and write  
\begin{align*}
\PP(|X_t|>\eps)&=\PP\bigg(\Big|tb(\eps)+M_t(\eps)+\sum_{i=1}^{N_t(\eps)}Z_i\Big|>\eps\bigg).
\end{align*}
Decomposing on the values of the Poisson process $N(\eps)$ leads to
{\small\begin{align}\label{eq:strategy}
 |\PP(|X_t|>\eps)-\lambda_\eps t|&\leq \PP(|tb(\eps)+M_t(\eps)|>\eps) +\lambda_\eps t|\PP(|tb(\eps)+M_t(\eps)+Z_1|>\eps) e^{-\lambda_\eps t}-1|+\PP(N_{t}(\eps)\ge 2).
\end{align}}
The last term raises no difficulty as $\PP(N_{t}(\eps)\ge 2)=O(\lambda_{\eps}^{2}t^{2})$. The first term is treated in Theorems \ref{ps1}  and \ref{ps2} which are established using  decomposition   \eqref{eq:x} at level $\eps/2$ and Lemma \ref{sj}. The proof of Theorem \ref{ps1} is made particularly technical by the presence of the drift term $b(\eps)$. This is the reason why, in the infinite variation counterpart Theorem \ref{ps2} we specialize to the symmetric case, hence $b(\eps)=0$.
Finally, to prove Theorems \ref{teo1} and \ref{lambda2} (resp. Theorem \ref{lambda2bis}) it remains to show:  $\PP(|tb(\eps)+M_t(\eps)+Z_1|\leq\eps)=O(t\lambda_\eps)$ (resp. $O_{\eps}(t^{1/\alpha})$) which corresponds to proving that
$|\PP(|tb(\eps)+M_t(\eps)+Z_1|>\eps) e^{-\lambda_\eps t}-1|=O(\lambda_\eps t)$  (resp. $O_{\eps}(t^{1/\alpha})$).

For this term the cases of finite variation (Theorem \ref{teo1}) and infinite variation (Theorems \ref{lambda2bis} and \ref{lambda2}) L\'evy processes essentially differ. For finite variation L\'evy processes, $\alpha\in(0,1)$, the result $\PP(|tb(\eps)+M_t(\eps)+Z_1|\leq\eps)=O(t\lambda_\eps)$ holds true and a main difficulty here lies in the management of the drift that can be nonzero. For infinite variation L\'evy processes, $\alpha\in[1,2)$, this result is not true in general. For instance, consider the case of a Cauchy process $X$ and fixed $\eps$.  The Cauchy process has a L\'evy density $(\pi x^{2})^{-1}\mathbf{1}_{\R\setminus\{0\}}$ and is therefore in $\mathscr L_{1/\pi,1}\cap\mathscr L_{1/\pi}$ and is  $\pi^{-1}2^{7}3^{-3}(\eps\wedge 1)^{-3}$-Lipschitz on the interval $((3/4(\eps\wedge1),2\eps-3/4(\eps\wedge1))$ for all $\eps>0$. Theorems \ref{ps2}, \ref{lambda2bis}, and \ref{lambda2} thus apply. For this example direct calculations allow to show that  $|\PP(|X_t|>\eps)-\lambda_\eps t|=O(\lambda_\eps^3 t^3)$ (see Section \ref{sec:ex}), however   $\lim_{t\to 0}\frac{\PP(|tb(\eps)+M_t(\eps)+Z_1|\leq\eps)}{t}=\infty$ implying that $\PP(|tb(\eps)+M_t(\eps)+Z_1|\leq\eps)=O(t\lambda_\eps)$ cannot hold. Indeed, the L\'evy measure being symmetric it leads to $b(\eps)=0$ and 
$$\PP(|M_t(\eps)+Z_1|\leq \eps)=\frac{1}{\lambda_\eps}\int_{-\infty}^{-\eps}\PP(|M_t(\eps)+z|\leq\eps) \frac{dz}{\pi z^2}+\frac{1}{\lambda_\eps}\int_{\eps}^{\infty}\PP(|M_t(\eps)+z|\leq\eps) \frac{dz}{\pi z^2}.$$
Fatou Lemma, joint with $\lim_{t\to 0}\frac{\PP(M_t(\eps)\in A)}{t}=\nu_\eps(A)$, $\nu_\eps=\nu\1_{|x|\leq \eps}$, and $f$ being symmetric, gives
\begin{align*}
\lambda_\eps \liminf_{t\to 0} \frac{\PP(|M_t(\eps)+Z_1|\leq \eps)}{t}&\geq \Big(\int_{-\infty}^{-\eps}+\int_{\eps}^\infty\Big)\liminf_{t\to 0}\frac{\PP(M_t(\eps)\in(-\eps-z,\eps-z))}{t}\frac{dz}{\pi z^2}\\
&\ge\int_{\eps}^\infty \nu_\eps(z-\eps,z+\eps)\nu(dz)=\int_{\eps}^{2\eps} \nu_\eps(z-\eps,\eps)\nu(dz)\\
&=\frac{1}{\pi^{2}}\int_{\eps}^{2\eps} \frac{2\eps-z}{\eps(z-\eps)}\frac{dz}{z^2}=\infty.
\end{align*}
We derive that the decomposition \eqref{eq:strategy} that leads to Theorem \ref{teo1}, $\alpha\in(0,1)$, does not permit to obtain optimal results for $\alpha\in[1,2)$ such as Theorem \ref{lambda2}. This is instead obtained by firstly adding a regularity assumption in a neighborhood of $\eps$ and secondly modifying the decomposition \eqref{eq:strategy}, considering a cutoff level $\eps'<\eps$, for example $\eps'=3\eps/4$ (see Lemmas \ref{psalpha} and \ref{yalpha} below).

Generalizing the results of Theorems  \ref{ps2}, \ref{lambda2bis} and \ref{lambda2} to non-symmetric L\'evy processes is possible at the expense of more cumbersome proofs and modifying the conditions on $t$.

\subsection{Examples\label{sec:ex}}
We consider four examples of L\'evy processes for which explicit formulas for their laws are available. This permits to conduct direct computations and expansions for the marginal laws and allows to compare them with the previous results. Let us stress that even in these cases where the law of the process is known, we do not know the law of the process corresponding to its small jumps. Besides the compound Poisson process, it is hard to propose examples to compare with Theorems \ref{ps1} and \ref{ps2}.
 Finally, we present a non-asymptotic control of the marginal law of $\alpha$-stable type processes. Proofs are postponed to Section \ref{proof:ex}.
\begin{enumerate}
\item  Let $X$ be a \textbf{compound Poisson process}. Then, for any $\eps>0$ 
$$
\big|\PP(|X_{t}|>{\varepsilon})-\lambda_{\varepsilon}t \big| = O_{\eps}\big(t^2\big)\quad\text{and}\quad \PP(|M_{t}(\eps\wedge 1)+tb(\eps\wedge 1)|>{\varepsilon\wedge 1}) = O_{\eps}\big(t^2\big),\quad \text{ as }t\to 0.
$$
It is possible to build examples for which these rates are sharp (see Section \ref{proof:ex}).
\item
Let $X$ be a  \textbf{Gamma process} of parameter $(1,1)$, that is a finite variation L\'evy process with L\'evy density $f(x)=\frac{e^{-x}}{x}\1_{(0,\infty)}(x)$, $\lambda_\varepsilon=\int_\varepsilon^\infty \frac{e^{-x}}{x}dx$ and $$\PP(|X_{t}|>\eps)=\PP(X_t>\varepsilon)=\int_\varepsilon^\infty \frac{x^{t-1}}{\Gamma(t)}e^{-x}dx,\quad \forall \eps>0,$$ where $\Gamma(t)$ denotes the $\Gamma$ function, i.e. $\Gamma(t)=\int_{0}^\infty x^{t-1}e^{-x}dx$. Then,  $$
\big|\PP(X_{t}>{\varepsilon})-\lambda_{\varepsilon}t \big| = O_{\eps}\big(t^2\big),\quad \text{ as }t\to 0.
$$

\item 
Let $X$ be an \textbf{inverse Gaussian process} of parameter $(1,1)$, i.e.
$$f(x)=\frac{e^{-x}}{x^{\frac{3}{2}}}\1_{(0,\infty)}(x)\quad \text{ and }\quad \PP(X_t>\varepsilon)=t e^{2t\sqrt \pi}\int_\varepsilon^\infty\frac{e^{-x-\frac{\pi t^2}{x}}}{x^{\frac{3}{2}}}dx,\quad \forall \eps>0.$$
Then,
\begin{equation}\label{IGP}
\big|\PP(|X_t|>\varepsilon)-t\lambda_{\varepsilon}\big|=O_{\eps}\big(t^2\big), \quad \text{ as }t\to0.
\end{equation}

\item \textbf{Cauchy processes.}
Let $X$ be a $1$-stable L\'evy process with
$$f(x)=\frac{1}{\pi x^2}\1_{\R\setminus \{0\}}\quad \text{ and }\quad \PP(|X_t|>\varepsilon)=2\int_{\frac{\varepsilon}{t}}^{\infty}\frac{dx}{\pi(x^2+1)},\quad \forall\eps>0.$$
Then, 
\begin{equation}\label{eq:cauchy}
 \big|\PP(|X_t|>\varepsilon)-t\lambda_\varepsilon\big|=O_{\eps}\big(t^3\big),\quad \text{ as }t\to 0.
\end{equation} 
For this example, the bound of Theorem \ref{lambda2} is suboptimal. However, improving Theorem \ref{lambda2} relying on the same strategy of proof, \emph{i.e.} using compound Poisson approximations, is hopeless and a different approach should be considered.

\item \textbf{$\alpha$-stable type processes.} 
Results for the cumulative distribution function for $\alpha$-stable  processes were already known (see \textit{e.g.} \cite{marchal2009small}). The following result is a generalization to any L\'evy process whose L\'evy measure behaves as an $\alpha$-stable  process in a neighborhood of the origin such as a tempered stable L\'evy prcess (see \textit{e.g.}  \cite{tankov} Section 4.2 or \cite{rosinski2007tempering}).  \end{enumerate}

\begin{corollary}\label{cor}
Let $X$ be a symmetric L\'evy process with a L\'evy measure $\nu$ absolutely continuous with respect to the Lebesgue measure and denote by $f=\frac{d\nu}{dx}$. Suppose that there exist $\alpha\in(0,2)$, $M_1>0$ and $M_2>0$ such that $M_1|x|^{-(1+\alpha)}\leq  |f(x)|\leq M_2|x|^{-(1+\alpha)}$, for all $0<|x|\leq 2$. Let $\eps\in(0,1]$ and $t>0$. We have:  
\begin{itemize}
\item If $\alpha\in(0,1):$ 
there exists a constant $\mathbf A_{M_1,M_2,\alpha} >0$, only depending on $M_1$, $M_2$ and $\alpha$, such that
\begin{equation*}
|\PP(|X_t|>\eps)-\lambda_\eps t|\leq \mathbf A_{M_1,M_2,\alpha} t^2\lambda_\eps^2,\quad \forall \ t\lambda_\eps\leq 2^{-\alpha}(2-\alpha)\alpha^{-1}.
\end{equation*}
\item If $\alpha\in [1,2)$ and $f\in \mathscr L_{M_2}$: there exist two constants $\mathbf B_{M_1,M_2,\alpha} >0$ and $\tilde B$, only depending on $M_1$, $M_2$ and $\alpha$, such that $\forall \ t\lambda_\eps\leq 2^{1-\alpha}M_2(1\wedge(2-\alpha)/2M_2)\alpha^{-1}$ it holds
\begin{equation*}
|\PP(|X_t|>\eps)-\lambda_\eps t|\leq \mathbf B_{M_1,M_2,\alpha} t^{1+1/\alpha}\lambda_\eps^{1+1/\alpha}\Big(\1_{\alpha\in(1,2)}+\ln\Big(\frac{\tilde B}{\lambda_\eps t}\Big)\1_{\alpha=1}\Big).
\end{equation*} 
\item If $\alpha\in [1,2)$ and $f$ is globally $M\eps^{-(2+\alpha)}$-Lipschitz on the interval $((3/4\eps,2\eps-3/4\eps)$: there exists a constant $\mathbf C_{M_1,M_2,\alpha} >0$, only depending on $M_1$, $M_2$ and $\alpha$, such that
\begin{equation*}
|\PP(|X_t|>\eps)-\lambda_\eps t|\leq \mathbf C_{M_1,M_2,\alpha} t^2 \lambda_\eps^2,\quad \forall \ t\lambda_\eps\leq 2^{-\alpha}(2-\alpha)\alpha^{-1}.
\end{equation*}
\end{itemize}
\end{corollary}

This result is a  consequence of Theorems \ref{teo1},  \ref{lambda2bis} and \ref{lambda2}  observing that, under the assumptions of Corollary \ref{cor}, 
\begin{align*}
2M_{1}\eps^{-\alpha}/\alpha\le\lambda_{\eps,1}\le 2M_{2}\eps^{-\alpha}/\alpha,\quad
\eps^\alpha\leq \frac{2M_2}{\alpha\lambda_{\eps,1}}\quad \text{and}\quad \eps^{-\alpha}\leq \frac{\alpha\lambda_{\eps}}{2M_1}.
\end{align*}

\subsection{Extension\label{sec:ext}}

A natural question is whether the above results hold true  for general L\'evy processes, that is in presence of a Gaussian part, $\Sigma>0$ in \eqref{eq:Xito1}. The answer is essentially positive but to avoid cumbersome proofs we chose to have $\Sigma=0$. If $\Sigma >0$, proofs can be adapted following the same steps as in Section \ref{sec:prf} replacing $M_{t}(\eps)$ with $\Sigma W_{t}+M_{t}(\eps)$, leading to similar results to those presented in Section \ref{Main}.

More precisely, in order to mimic what is done in Section \ref{sec:prf} for pure jump L\'evy processes, we need to generalize Lemma \ref{sj}. Adapting its proof we obtain the following result. 
For any $\eps\in(0,1]$, $t>0$ and $x>0$, it holds:
$$\PP(\Sigma W_t+M_t(\eps)>x)\leq e^{\frac{x}{\eps}}\bigg(\frac{t\sigma^2(\eps)}{x\eps+t\sigma^2(\eps)}\bigg)^{\frac{x\eps+t\sigma^2(\eps)}{\eps^2}}\exp\Big({t\frac{\Sigma^2}{2\eps^2} \log^2\big(1+\frac{x\eps}{t\sigma^2(\eps)}\big)}\Big).$$
In particular, using that $u\mapsto u\log^{2}(1+1/u)$ is bounded by 1 for $u>0$, we observe that  the additional term $e^{t\frac{\Sigma^2}{2\eps^2} \log^2\big(1+\frac{x\eps}{t\sigma^2(\eps)}\big)}\le e^{\frac{\Sigma^{2}x}{2\eps\sigma^{2}(\eps)}}$ is bounded.\\

Similarly, it is possible to have a more general drift $b$ in the triplet (see \eqref{eq:Xito1}). Proofs can be adapted at the cost of a more stringent condition on $t$. Indeed, the condition on $t$ in the  above Theorems  ensures that $tb(\eps)\le \eps/2$, a similar condition should be satisfied in presence of a general  drift $b$.

\section{Proofs\label{sec:prf}}

\subsection{Preliminaries}
Introduce the following notations. Consider $b\ge a>0$, denote by $\lambda_{a}:=\int_{|x|>a}f(x)dx$ and $\lambda_{a,b}:=\int_{b>|x|>a}f(x)dx$ with the convention $\lambda_{a,a}=0.$  Recall that  $\sigma^{2}(a):=\int_{0<|x|<a}x^{2}f(x)dx$ and for finite variation processes the drift is denoted by $b(a):=\int_{0< |x|<a}xf(x)dx$. Furthermore, we write $Y^{(a)}$ (resp. $Y^{(a,b)}$) for a random variable with density $f\mathbf{1}_{(-a,a)^{c}}/\lambda_{a}$ (resp. $f\mathbf{1}_{[-b,-a]\cup[a,b]}/\lambda_{a,b}$). With these notations, following \eqref{eq:x} consider the decomposition which plays an essential role in the sequel, for all $t>0$
\begin{align}
\label{eq:decprf}M_t(\eps)=M_t(\eta)+Z_t(\eta,\eps)-t\big(b(\eps)-b(\eta)\big), \quad \forall\ 0<\eta< \eps\leq 1,
\end{align} where $Z_t(\eta,\eps)=\sum_{i=1}^{N_{t}(\eta,\eps)}Y_{i}^{(\eta,\eps)}$, $N(\eta,\eps)$ being a Poisson process of intensity $\lambda_{\eta,\eps}$ independent of $(Y_{i}^{(\eta,\eps)})$.
Therefore, for all $0<x\leq \delta$ and $t>0$ it holds:
\begin{align}\label{poisson}
\PP(N_t(x,\delta)\geq 1)\leq \lambda_{x,\delta} t \quad \text{ and }\quad \PP(N_t(x,\delta)\geq 2)\leq (\lambda_{x,\delta} t )^2.
\end{align}

In the sequel we make intensive use of the following inequalities. 
For any $0<x\le y\leq 2$ and $f$ in $\mathscr L_{M,\alpha}$, it holds
\begin{align}
\frac{\sigma^{2}(x)}{x^{2}}&=\frac{\int_{-x}^x u^2 f(u)du}{x^2}\leq \frac{2M}{2-\alpha} x^{-\alpha},\label{eq:prfsigma}\\
\lambda_{x,y}&= \int_{y>|u|>x} f(u)du \leq \frac{2M}{\alpha} x^{-\alpha},\label{eq:prflbd}\\
b(x)&= \int_{|u|\leq x} uf(u)du \leq \frac{2M}{1-\alpha} x^{1-\alpha}.\label{eq:prfb}
\end{align}
\subsection{Proof of Theorem \ref{ps1} }
First, note that
\begin{align*}
\PP(&|tb(\eps)+M_{t}(\eps)|> \eps)=\PP(tb(\eps)+M_{t}(\eps)> \eps)+\PP(tb(\eps)+M_{t}(\eps)<- \eps).
\end{align*}
We consider only the term $\PP(tb(\eps)+M_{t}(\eps)> \eps)$ as $\PP(tb(\eps)+M_{t}(\eps)<- \eps)$ can be treated analogously. 
Define 
\begin{equation*}
\eta:=\inf \bigg\{\frac{\eps}{4}\leq u<\eps:\ u\leq \frac{\eps-t\int_{-u}^u x f(x)dx}{2},t\lambda_{\eps/8,u}<2\bigg\}.
\end{equation*}
Observe that if $f\in\mathscr L_{M,\alpha}$, $M>0$, $\alpha\in(0,1)$, $\eps\in(0,1]$ and $0<t\leq (1-\alpha) M^{-1}( \eps/4)^{\alpha}$, then  the set 
$A_{\eps,t}:=\big\{\frac{\eps}{4}\leq u<\eps:\ u\leq \frac{\eps-tb(u)}{2},t\lambda_{\eps/8,u}<2\big\}$
is not empty as $\eps/4\in A_{\eps,t}$ noting in particular that $t\lambda_{\eps/8,\eps/4}\le 2(1-\alpha)(2^\alpha-1)/\alpha\le 2\log(2)$. 

By means of \eqref{eq:decprf} and the definition of $b(\cdot)$, we have 
\begin{align}\label{eq:decProp1}
\PP(t&b(\eps)+M_{t}(\eps)> \eps)
=\PP\big(M_{t}(\eta)+Z_{t}(\eta,\eps)> \eps-tb(\eta)\big)\nonumber \\
&\leq \PP(M_t(\eta)>\eps-tb(\eta))+\lambda_{\eta,\eps}t \PP\big(M_t(\eta)+Y_{1}^{(\eta,\eps)}>\eps-tb(\eta))+ \PP\big(N_t(\eta,\eps)\geq 2\big) ,\end{align} where we decomposed on the values of the Poisson process $N(\eta,\eps)$.
Using \eqref{poisson}, we have $\PP\big(N_{t}(\eta,\eps)\geq 2\big)\leq (\lambda_{\eta,\eps}t)^{2}$. We thus only have to control the first and second addendum in \eqref{eq:decProp1}.
For the first one, we apply Lemma \ref{sj}, using that $t\leq (1-\alpha) M^{-1} \eps^{\alpha}4^{-(1+\alpha)}$ implies that $t\sigma^2(x)x^{-2}\leq 1$ for all $x\in[\eps/4,\eps]$. It follows from the definition of $\eta$ and \eqref{sg} that
\begin{align*}
\PP(M_{t}(\eta)>\eps-tb(\eta))&\le \PP(M_{t}(\eta)>2\eta)\leq \bigg(\frac{e\sigma^2(\eta)}{4\eta^2}\bigg)^{2} e^{e^{-1}}t^{2}.
\end{align*}
Hence, using \eqref{eq:prfsigma} and the fact that $\eta\geq \eps/4$ and $4^{2\alpha-1}e^{2+1/e}(2-\alpha)^{-2}\le 16$, leads to
\begin{align}\label{eq:decT1}
\PP(M_{t}(\eta)>\eps-tb(\eta))&\leq 16 t^2 M^2 \eps^{-2\alpha}.
\end{align}
For the second term in \eqref{eq:decProp1}, set $\eps':=\eps-tb(\eta)$ and notice that $\eps'\geq \eps/2$. It holds
\begin{align*}
\lambda_{\eta,\eps}&\PP\big(M_t(\eta)+Y_{1}^{(\eta,\eps)}>\eps')=\int_{\eta<|y|<\eps}\PP\big(M_t(\eta)>\eps'-y) f(y)dy\\
&\leq \PP\big(M_t(\eta)>\eps'+\eta)\int_{-\eps}^{-\eta} f(x)dx+\int_{\eta<y<\eps}\PP\big(M_t(\eta)>\eps'-y) f(y)dy=: T_1+T_2.
\end{align*}
From $\eps'>0$ it follows that $ \PP\big(M_t(\eta)>\eps'+\eta)\leq  \PP\big(M_t(\eta)>\eta)$. The Markov inequality and \eqref{eq:prfsigma}, joined with the fact that $f\in\mathscr L_{M,\alpha}$ and $\eta\ge\eps/4$ yield
\begin{align}\label{t1}
T_1&\leq 2M^2 t\eta^{-\alpha}(2-\alpha)^{-1}\int_{\eta}^\eps|x|^{-(1+\alpha)}dx\leq \frac{2M^2}{\alpha(2-\alpha)}  \eta^{-2\alpha} t\leq tM^2 \eps^{-2\alpha} \mathbf C_{1,\alpha},
\end{align}
with
\begin{equation*}
\mathbf C_{1,\alpha}:= \frac{2^{1+4\alpha}}{\alpha(2-\alpha)}.
\end{equation*}

To treat the term $T_2$ we suppose that $b(\eta)\geq0$, the case $b(\eta)<0$ is handled similarly. After a change of variable, we obtain
\begin{align*}
T_2&=\int_{-tb(\eta)}^{\eps'-\eta} \PP(M_t(\eta) >x) f(\eps'-x)dx\\
&\leq \int_{-tb(\eta)}^{0}  f(\eps'-x)dx+\int_{0}^{\eta/2} \PP(M_t(\eta) >x) f(\eps'-x)dx\\
&\quad+\int_{\eta/2}^{\eta} \PP(M_t(\eta) >x) f(\eps'-x)dx+\int_{\eta}^{\eps'-\eta} \PP(M_t(\eta) >x) f(\eps'-x)dx \\
&=:T_{2,1}+T_{2,2}+T_{2,3}+T_{2,4}.
\end{align*}
First observe that for $f\in\mathscr L_{M,\alpha}$ and $\eps'\ge\eps/2$ we get
$$f(\eps'-x)\leq \frac{M}{|\eps'-x|^{1+\alpha}}\leq M(\eps')^{-(1+\alpha)}\le M 2^{1+\alpha}\eps^{-(1+\alpha)},\quad \forall x\in[-tb(\eta),0].$$
Furthermore, using that $b(\eta)\leq 2M(1-\alpha)^{-1} \eta^{1-\alpha}\leq 2M(1-\alpha)^{-1} \eps^{1-\alpha}$, we conclude that
\begin{equation}\label{t21}
T_{2,1}\leq  \frac{2^{2+\alpha}tM^2}{1-\alpha}\eps^{-2\alpha}.
\end{equation}
Next we consider $T_{2,2}$. By \eqref{eq:decprf}, for any $\tilde x\in(0,\eta)$, we write $M_t(\eta) =M_t(\tilde x)+Z_t(\tilde x,\eta)-t(b(\eta)-b(\tilde x))$. 
Consider $x\in(2Mt\eta^{1-\alpha} (1-\alpha)^{-1},\eta/2)$ and set $\tilde x:=x-2Mt\eta^{1-\alpha} (1-\alpha)^{-1}$. 
Observe that, as $0<t\leq (1-\alpha) M^{-1} \eps^{\alpha}4^{-(1+\alpha)}$ it holds $2Mt\eta^{1-\alpha} (1-\alpha)^{-1}\leq \eta/2$.
Using that $f\in\mathscr L_{M,\alpha}$ we have:$$|b(\eta)-b(\tilde x)|=\bigg|\int_{|u|\in[\tilde x,\eta]} u f(u)du\bigg|\leq 2M \eta^{1-\alpha}(1-\alpha)^{-1}$$
from which we derive that $\PP(M_t(\tilde x) >x+t(b(\eta)-b(\tilde x)))\leq \PP(M_t(\tilde x)>\tilde x)$.
It follows that for $x\in(2Mt\eta^{1-\alpha} (1-\alpha)^{-1},\eta/2)$ we may write, decomposing on the values of $N(\tilde x,\eta)$, that
\begin{align*}
 \PP(M_t(\eta) >x)&= \PP\big(M_t(\tilde x)+Z_t(\tilde x,\eta) >x+t(b(\eta)-b(\tilde x))\big)\\
 &\leq  \PP\big(M_t(\tilde x) >\tilde x\big)+\PP(N_t(\tilde x,\eta)\geq 1)\\
 &\leq t\frac{2M}{2-\alpha}(\tilde x)^{-\alpha} +t\lambda_{\tilde x}\leq \frac{2Mt (\tilde x)^{-\alpha}(2+\alpha)}{\alpha(2-\alpha)},
 \end{align*}
 where, in the last inequality, we used the Markov inequality and \eqref{eq:prfsigma}.
Consequently, using that $\eta\le\eps$ and noticing that $3/8\eps\leq \eps'-x\leq 1$ for all $x\in(0,\eta/2)$, we derive
\begin{align}\label{t22}
T_{2,2}&\leq \int_0^{\frac{2Mt\eta^{1-\alpha}}{ 1-\alpha}}f(\eps'-x)dx+ \frac{2(2+\alpha)Mt  }{\alpha(2-\alpha)}\int_{\frac{2Mt\eta^{1-\alpha}}{ 1-\alpha}}^{\eta/2} \Big(x-\frac{2Mt\eta^{1-\alpha}}{ 1-\alpha}\Big)^{-\alpha}f(\eps'-x)dx\nonumber\\ \nonumber
&\leq \frac{2M^2t\eps^{-2\alpha}}{ 1-\alpha}\Big(\frac{8}{3}\Big)^{1+\alpha}+ \frac{2(2+\alpha)M^2t  }{2^{1-\alpha}\alpha(2-\alpha)(1-\alpha)}\Big( \frac{8}{3}\Big)^{1+\alpha} \eps^{-2\alpha}\\
&\leq \frac{2M^2t\eps^{-2\alpha}}{ 1-\alpha}\Big(\frac{8}{3}\Big)^{1+\alpha}\bigg(1+\frac{2+\alpha}{2^{1-\alpha} \alpha(2-\alpha)}\bigg).
\end{align}

To treat the term $T_{2,3}$ we proceed analogously. Let $x\in[\eta/2,\eta]$ and $\tilde Z_t(x,\eta)$ be a centered version of $Z_t(x,\eta)$, that is 
$\tilde Z_t(x,\eta)=\sum_{i=1}^{N_t(x,\eta)}\big(Y_i^{(x,\eta)}-\E[Y_i^{(x,\eta)}]\big)$. In particular, by definition of $\eta$, if follows that $t\lambda_{x,\eta}<2$ and  Lemma \ref{cCPP} applies. On the one hand we derive that
 \begin{align*}
 |&\PP(M_t(x)+Z_t(x,\eta)-\E[Z_t(x,\eta)] >x)-\PP(M_t(x)+\tilde Z_t(x,\eta) >x)|\\
 &\leq 4t\lambda_{x,\eta}|\E[Y_1^{(x,\eta)}]|\sup_{|y|\in[x,\eta]}|f(y)/\lambda_{x,\eta}|\leq tM 2^{2+\alpha}\eta^{-(1+\alpha)}\int_{x<|u|<\eta}\frac{\eta f(u)}{\lambda_{(x,\eta)}}  du\leq  2^{2+\alpha}tM \eta^{-\alpha}, \end{align*} 
 where we used that $\E[Z_t(x,\eta)]=t(b(\eta)-b(x)).$ 
  On the other hand, we have that
 \begin{align*}
 \PP(M_t(x)+\tilde Z_t(x,\eta) >x)&\leq  \PP(M_t(x) >x)+ \PP(N_t(x,\eta) \geq 1)
\leq \frac{10Mt  x^{-\alpha}}{\alpha(2-\alpha)}\leq \frac{20Mt  \eta^{-\alpha}}{\alpha(2-\alpha)},
 \end{align*}
 where we used the Markov inequality, \eqref{eq:prfsigma}, \eqref{poisson}, \eqref{eq:prflbd} and that $x>\eta/2$.
Finally, by the triangle inequality and using that $\eps'-\eta\geq \eta\geq \eps/4$, we deduce that
\begin{align}\label{t23}
T_{2,3}&\leq \frac{28Mt  \eta^{-\alpha}}{\alpha(2-\alpha)}\int_{\eta/2}^\eta f(\eps'-x)dx\leq   \frac{28M^2t  \eta^{-\alpha}(\eps'-\eta)^{-\alpha}}{\alpha^2(2-\alpha)}\leq  \frac{28\times 4^{2\alpha}M^2t  \eps^{-2\alpha}}{\alpha^2(2-\alpha)}.
\end{align}
Then, for the term $T_{2,4}$, the Markov inequality and \eqref{eq:prfsigma}, for any $x\in[\eta,\eps'-\eta]$, lead to
$$\PP(M_t(\eta)>x)\leq \frac{2M}{2-\alpha}\eta^{-\alpha}t. $$
Therefore, using that $\eps'-\eta\geq \eta\geq \eps/4$, we get
\begin{align}\label{t24}
T_{2,4}&\leq \frac{2M}{2-\alpha}\eta^{-\alpha}t \int_\eta^{\eps'-\eta}f(\eps'-x)dx\leq \frac{2M^2}{(2-\alpha)\alpha}\eta^{-2\alpha}t\leq  \frac{2^{1+4\alpha}M^2}{\alpha(2-\alpha)}\eps^{-2\alpha}t.
\end{align}
Gathering Equations \eqref{t21}, \eqref{t22},  \eqref{t23} and  \eqref{t24} yield 
\begin{align}\label{t2}
T_2&\leq tM^2 \eps^{-2\alpha} \mathbf C_{2,\alpha},
\end{align}
with 
\begin{align*}
\mathbf C_{2,\alpha}=\bigg(\frac{2^{2+\alpha}}{1-\alpha}+\frac{ 2^{4\alpha+3}(2^{1-\alpha}\alpha(2-\alpha)+2+\alpha)}{\alpha(2-\alpha)(1-\alpha)3^{1+\alpha}}+\frac{28\times 4^{2\alpha}}{\alpha^2(2-\alpha)}+\frac{2^{1+4\alpha}}{\alpha(2-\alpha)}\bigg).
\end{align*}
Combining \eqref{t1} and \eqref{t2} we conclude that, if $b(\eta)\geq 0$, then 
\begin{equation}\label{eq:decT2}
\lambda_{\eta,\eps}t\PP\big(M_t(\eta)+Y_{1}^{(\eta,\eps)}>\eps')\leq  t^2M^2 \eps^{-2\alpha} (\mathbf C_{1,\alpha}+\mathbf C_{2,\alpha}).
\end{equation} The case $b(\eta)<0$ is treated similarly and therefore not detailed here.
Injecting in \eqref{eq:decProp1} Equations \eqref{eq:decT1}, \eqref{eq:prflbd} and \eqref{eq:decT2} we conclude that
\begin{equation}\label{ctealpha}\PP(tb(\eps)+M_t(\eps)>\eps)\leq t^2 M^2\eps^{-2\alpha} (16+64\alpha^{-2}+\mathbf C_{1,\alpha}+\mathbf C_{2,\alpha})=: t^2 M^2\eps^{-2\alpha} \mathbf   {\rm C}_1,
\end{equation}
as desired.

\medskip

For a symmetric L\'evy measure above computations can be simplified. In this case $b(\eps)=0$ and one can directly take $\eta=\eps/2$ in the previous lines. More precisely, it holds
\begin{align*}
\PP(M_t(\eps)>\eps)\leq \PP(M_t(\eps/2>\eps)+(t \lambda_{\eps/2,\eps})^2+  t \lambda_{\eps/2,\eps}\PP(M_t(\eps/2)+Y_1^{(\eps/2,\eps)}>\eps).
\end{align*} 
To control the first two addendum use Lemma \ref{sj} and \eqref{eq:prflbd}. To treat the last term we proceed as follows:
\begin{align*}
\lambda_{\eps/2,\eps}\PP(M_t(\eps/2)&+Y_1^{(\eps/2,\eps)})=\int_{\eps/2}^\eps +\int_{-\eps}^{-\eps/2} \PP(M_t(\eps/2)>\eps-z)f(z) dz\\
&\leq \int_{\eps/2}^\eps \big(\PP(M_t(\eps-z)>\eps-z) +t\lambda_{\eps/2,\eps}\big) f(z) dz+\frac{\PP(M_t(\eps/2)>3/2\eps)}{2}\lambda_\eps\\
&\leq t \int_{\eps/2}^\eps \Big(\frac{\sigma^2(\eps-z)}{(\eps-z)^2} +\lambda_{\eps/2,\eps}\Big) f(z) dz+\frac{\PP(M_t(\eps/2)>3/2\eps)}{2}\lambda_{\eps/2,\eps}\\
&\leq \frac{4^{1+\alpha} tM^2\eps^{-2\alpha}}{\alpha(1-\alpha)(2-\alpha)}+\frac{\PP(M_t(\eps/2)>3/2\eps)}{2}\lambda_{\eps/2,\eps}.
\end{align*}
The term $\PP(M_t(\eps/2)>3/2\eps)$ is controlled applying Lemma \ref{sj} using that $4t\sigma^2(\eps/2)\leq \eps^2$. Collecting all the pieces together, one derives the following result: For all $t>0$ such that $t\leq \eps^\alpha(2-\alpha)M^{-1}2^{-\alpha-1}$ (implying that  $t\lambda_{\eps/2,\eps}\leq 1$), it holds:
$\PP(M_t(\eps)>\eps)\leq t^2\eps^{-2\alpha}M^2  {\rm C}_2,$
where \begin{equation}\label{sjsym} 
  {\rm C}_2:= \frac{3 \times 2^{2\alpha-1} e^{2+1/e}}{(2-\alpha)^2}+\frac{4^{1+\alpha}}{\alpha(1-\alpha)(2-\alpha)}+\frac{4^\alpha}{\alpha^2}.
\end{equation}

\subsection{Proof of Theorem \ref{teo1}}

To prove Theorem \ref{teo1} we first introduce an auxiliary result.
\begin{lemma}\label{aux}
Let $\nu$ be a L\'evy measure with density $f$ with respect to the Lebesgue measure and $\eps$ a positive real number. Set $\rho:=\eps\wedge 1$ and 
$$Q:=|\lambda_\rho t \PP(|M_t(\rho)+tb(\rho)+Y_1^{(\rho)}|>\eps)-\lambda_\eps t |.$$
\begin{itemize}
\item If $\eps\in(0,1]$ and $f\in\mathscr L_{M,\alpha}$ for some $\alpha\in(0,1)$ and $M>0$, then
\begin{equation*}
Q\leq t^{2}(M^2{\rm D}_1\eps^{-2\alpha}+M\lambda_\eps\eps^{-\alpha}
{\rm D}_2),\quad \forall\ 0<t< (1-\alpha) M^{-1} \eps^{\alpha}4^{-(1+\alpha)},
\end{equation*}
where $\rm D_1$ and $\rm D_2$ are defined as in \eqref{cte3}.
\item If $\eps>1$ and $f\in\mathscr L_{M,\alpha}\cap \mathscr L_M$ for some $\alpha\in(0,1)$ and $M>0$, then for all $0<t<(1-\alpha)(5M)^{-1}$ it holds
\begin{align*}
Q&\leq 2M^2t^2\Big(\tilde{\rm D}_1+\frac{4}{2-\alpha}(\eps-3/2-t|b(1)|)\1_{\eps>3/2+t|b(1)|}
\Big)\\
&\quad +Mt^2\bigg(4\times 5^{\alpha}\1_{1<\eps<1+2t|b(1)|}+\frac{8}{5}+3\lambda_2+\frac{4\lambda_1}{2-\alpha}\bigg),
\end{align*}
where $\tilde{\rm D}_1$ is defined as in \eqref{tdlm3}.
\end{itemize}
If in addition we suppose that $\nu$ is a symmetric measure, then 
\begin{itemize}
\item If $\eps\in(0,1]$ and $f\in\mathscr L_{M,\alpha}$ for some $\alpha\in(0,1)$ and $M>0$, it holds
\begin{equation*}
Q\leq  \frac{t^2M}{2(2-\alpha)}\big(\lambda_\eps \eps^{-\alpha}+4\lambda_{2\eps} \eps^{-\alpha}\big)+2t^2M^2{\rm D_3} \eps^{-2\alpha},\quad \forall \ t>0,
\end{equation*}
where ${\rm D_3}$ is defined as in \eqref{d2lm3}.
\item If $\eps>1$ and $f\in\mathscr L_{M,\alpha}\cap \mathscr L_M$ for some $\alpha\in(0,1)$ and $M>0$, it holds
\begin{equation*}
Q\leq \frac{t^2M}{2-\alpha}\bigg(\lambda_1 2^{-\alpha}+\frac{4M}{\alpha(1-\alpha)}+\lambda_{1+\eps} \bigg),\quad \forall\ t>0.
\end{equation*}
\end{itemize}

 \end{lemma}

\paragraph{Proof of Theorem \ref{teo1}}
Using the decomposition $X_t=M_t(\rho)+tb(\rho)+Z_t(\rho)$, $\rho=\eps\wedge 1$, we derive, decomposing on the Poisson process $N(\rho)$, that
\begin{align*}
|\PP(|X_t|>\eps)-\lambda_\eps t|&= \bigg|\PP(|M_t(\rho)+tb(\rho)|>\eps)e^{-\lambda_\rho t}+\lambda_\rho t \PP(|M_t(\rho)+tb(\rho)+Y_1^{(\rho)}|>\eps) e^{-\lambda_\rho t}\\&\quad -\lambda_\eps t 
 + \sum_{n=2}^\infty\PP\bigg(|M_t(\rho)+tb(\rho)+\sum_{i=1}^{n}Y_i^{(\rho)}|>\eps\bigg)\PP(N_t(\rho)=n)\bigg|\\
&\leq \PP(|M_t(\rho)+tb(\rho)|>\rho)+|\lambda_\rho t \PP(|M_t(\rho)+tb(\rho)+Y_1^{(\rho)}|>\eps)-\lambda_\eps t |\\
&\quad + \lambda_\rho t (1-e^{-\lambda_\rho t}) + \PP(N_t(\rho)\geq 2)\\
:&=I_{1}+I_{2}+I_{3}+I_{4}.
\end{align*}
The term $I_{1}$ is controlled with Theorem \ref{ps1}, $I_{2}$ with Lemma \ref{aux}, for $I_{3}$ use that $1-e^{-x}\leq x$, for all $x>0$ to get $I_{3}\leq \lambda_{\rho}^{2}t^{2}$ and finally,  it follows from \eqref{poisson} that  $I_{4}=\PP(N_{t}(\rho)\geq 2)\leq \lambda_{\rho}^{2}t^{2}$ as ($1-e^{-x}-xe^{-x}\leq x^2$, for all $x>0$). 

\subsection{Proof of Theorem \ref{ps2}}
As $\nu$ is symmetric it holds $\PP(|M_t(\eps)|\geq \eps)=2\PP(M_t(\eps)\geq \eps)$. Using the same reasoning as in the proof of Theorem \ref{ps1} we get 
\begin{equation}\label{decps2}
\PP(M_t(\eps)\geq \eps)\leq \PP(M_t(\eps/2)\geq \eps)+t\lambda_{\eps/2,\eps} \PP(M_t(\eps/2)+Y_1^{(\eps/2,\eps)}\geq \eps)+(t\lambda_{\eps/2,\eps})^2.
\end{equation}
By means of Lemma \ref{sj} joined with \eqref{eq:prfsigma}, we get that 
\begin{equation}\label{m1}
\PP(M_t(\eps/2)\geq \eps)\leq t^2 \frac{4M^2e^{2+1/e}}{(2-\alpha)^2\eps^{2\alpha}},
\end{equation}
and, using \eqref{eq:prflbd}, that 
\begin{equation}\label{l1}
(t\lambda_{\eps/2,\eps})^2\leq \frac{t^2 M^2 4^{1+\alpha}}{\alpha^2\eps^{2\alpha}}.
\end{equation}
Finally, using the symmetry of $\nu$, we have that
\begin{align*}
 \lambda_{\eps/2,\eps}&\PP(M_t(\eps/2)+Y_1^{(\eps/2,\eps)}\geq \eps) =\int_{\eps/2}^\eps \big( \PP(M_t(\eps/2)\geq \eps -z)+ \PP(M_t(\eps/2)\geq \eps +z)\big)f(z)dz\nonumber \\
 &\leq \int_{\eps/2}^\eps  \PP(M_t(\eps/2)\geq \eps -z)f(z)dz+ \frac{ \PP(M_t(\eps/2)\geq 3/2\eps)}{2} \lambda_{\eps/2,\eps}
 =:T_1+T_2.\nonumber
\end{align*}
To control the term $T_1$, observe that 
\begin{align*}
T_1
&=\int_0^{t^{1/\alpha}} \PP(M_t(\eps/2)\geq z) f(\eps-z) dz+\int_{t^{1/\alpha}}^{\eps/2} \PP(M_t(\eps/2)\geq z) f(\eps-z) dz\\
&\leq \frac{Mt^{1/\alpha}}{(\eps-t^{1/\alpha})^{1+\alpha}}+\frac{2^{1+\alpha}M}{\eps^{1+\alpha}}\int_{t^{1/\alpha}}^{\eps/2} \PP(M_t(\eps/2)\geq z) dz.
\end{align*}
Next, for $z\in (t^{1/\alpha},\eps/2)$, the Markov inequality and \eqref{eq:prfsigma} lead to
\begin{align*}
 \PP(M_t(\eps/2)\geq z)& \leq \PP(M_t(z)\geq z)+ t\lambda_{z,\eps/2}\leq \frac{t\sigma^2(z)}{z^2} +2t M\int_{z}^{\eps/2} \frac{dx}{x^{1+\alpha}}\\
 &\leq 2M t z^{-\alpha}\Big(\frac{1}{2-\alpha}+\frac{1}{\alpha}\Big).
\end{align*}
Therefore, for any $\alpha\in(1,2)$
$$\int_{t^{1/\alpha}}^{\eps/2} \PP(M_t(\eps/2)\geq z) dz\leq \frac{4M t^{\frac{1}{\alpha}}}{\alpha(2-\alpha)(\alpha-1)}, $$
then, using that $\eps-t^{1/\alpha}\geq \eps/2$, we derive that
\begin{align}\label{t1ps2}
T_1\leq \frac{2^{1+\alpha}M t^{1/\alpha}}{\eps^{1+\alpha}}\bigg(1+\frac{M}{\alpha(2-\alpha)(\alpha-1)}\bigg),\quad \alpha\in(1,2).
\end{align}
If, instead, $\alpha=1$, we get 
\begin{align}\label{t1ps2alpha1}
T_1\leq \frac{4M t}{\eps^2}+\frac{16M^2}{\eps^2}t \ln\Big(\frac{\eps}{2t}\Big).
\end{align}
To control the term $T_2$ we use once again the Markov inequality joined with \eqref{eq:prfsigma} to obtain
\begin{align}\label{t2ps2}
T_2&\leq\frac{t\sigma^2(\eps/2)}{9(\eps/2)^2}\frac{\lambda_{\eps/2,\eps}}{2}\leq \frac{2^{\alpha+1}M^2 t}{9\alpha(2-\alpha)\eps^{2\alpha}}, \quad \alpha\in[1,2).
\end{align}
Gathering \eqref{t1ps2} and \eqref{t2ps2} we have, for $\alpha\in(1,2)$, 
\begin{align}\label{eq:p2}
 \lambda_{\eps/2,\eps}\PP(M_t(\eps/2)+Y_1^{(\eps/2,\eps)}\geq \eps)\leq \frac{2^{1+\alpha}M t^{1/\alpha}}{\eps^{1+\alpha}}\bigg(1+\frac{M}{\alpha(2-\alpha)(\alpha-1)}\bigg)+ \frac{2^{\alpha+1}M^2 t}{9\alpha(2-\alpha)\eps^{2\alpha}}.
\end{align}
Combining \eqref{decps2} with \eqref{m1}, \eqref{l1} and \eqref{eq:p2} we conclude that for all $\alpha\in(1,2)$ it holds
\begin{align*}
\PP(M_t(\eps)\geq \eps)\leq  \frac{t^2M^2}{\eps^{2\alpha}}{\rm E}_{1}+\frac{2^{1+\alpha}M t^{1+1/\alpha}}{\eps^{1+\alpha}}\bigg(1+\frac{M}{\alpha(2-\alpha)(\alpha-1)}\bigg),
\end{align*} with \begin{align}
\label{cteT3}{\rm E}_{1}:=\bigg(\frac{4e^{2+1/e}}{(2-\alpha)^2}+\frac{4^{1+\alpha}}{\alpha^2}+\frac{2^{\alpha+1}}{9\alpha(2-\alpha)}\bigg).
\end{align}
If, instead, $\alpha=1$, then using \eqref{t1ps2alpha1}
\begin{align*}
\PP(M_t(\eps)\geq \eps)\leq  \frac{4t^2M^2}{\eps^2}\bigg(e^{2+1/e}+ \frac{37}{9}\bigg)+\frac{4M t^2}{\eps^2}+\frac{16M^2}{\eps^2}t^2 \ln\Big(\frac{\eps}{2t}\Big).
\end{align*}
This concludes the proof.

\subsection{Proof of Theorem \ref{lambda2bis}}

 \begin{lemma}\label{yalphabis}
Let $\nu$ be a symmetric L\'evy measure with density $f$ with respect to the Lebesgue measure and $f\in \mathscr L_{M,\alpha}\cap \mathscr L_M$ for some $\alpha\in[1,2)$ and $M>0$. Let $\eps>0$ and set $\rho=\eps\wedge1$. Then, for all $0<t<(\eps{\wedge 1}/2)^\alpha\big(1\land ((2-\alpha)/2M)\big)$ it holds: 
 \begin{align*}
\big|\lambda_\rho t \PP(|M_t(\rho)+Y_1^{(\rho)}|&>\eps)-\lambda_\eps t \big|  \leq{\rm L}_{1}\frac{t^{1+1/\alpha}}{(\eps\wedge 1)^{1+\alpha}}+\frac{8M^{2}}{\alpha(2-\alpha)}\frac{t^{2}}{(\eps\wedge 1)^{2\alpha}}+\frac{5M}{2-\alpha}\frac{t^{2}\lambda_{1}}{(\eps\wedge1)^2}\\
&\quad+\frac{4M^{2}t^{2}}{2-\alpha}\1_{\eps>2}\eps +12M^2t\1_{\alpha=1}\ln\Big(\frac{C(1\wedge\eps\wedge(\eps-1\vee 0))}{t}\Big)\frac{1}{\eps\wedge 1} \end{align*}
where $C:=\big(1\wedge\big((2-\alpha)/2M\big)\big)^{1/\alpha}$ and ${\rm L}_{1}$ is defined in \eqref{eq:L123}
 \end{lemma}
 \begin{proof}[Proof of Theorem \ref{lambda2bis}]
 The result follows from Theorem \ref{ps2} and Lemma \ref{yalphabis} using the decomposition
 \begin{align*}
|\PP(|X_t|>\eps)-\lambda_\eps t|&\leq \PP(|M_t(\rho)|>\rho)+|\lambda_\rho t \PP(|M_t(\rho)+Y_1^{(\rho)}|>\eps)-\lambda_\eps t |+2\lambda_\rho^2t^2\\
&\leq{\rm G}_{1}\frac{t^{1+1/\alpha}}{(\eps\wedge 1)^{1+\alpha}}+{\rm G}_{2}\frac{t^{2}}{(\eps\wedge 1)^{2\alpha}}+\frac{5M}{2-\alpha}\frac{t^{2}\lambda_{1}}{(\eps\wedge1)^2}+\frac{4M^{2}t^{2}}{2-\alpha}\1_{\eps>2}\eps \\
&\quad+M^2t^{2}\1_{\alpha=1}\bigg(\frac{12}{\eps\wedge 1}\ln\Big(\frac{C(1\wedge\eps\wedge(\eps-1\vee 0))}{t}\Big)+\frac{16}{\eps^2} \ln\Big(\frac{\eps}{2t}\Big)\bigg)+2\lambda_{\eps\wedge1}^{2}t^{2},
\end{align*}
with $\rho:=\eps\wedge 1$ and
\begin{align}\label{eq:G1G2}
{\rm G}_{1}&={\rm L}_{1}+\mathbf{1}_{\alpha\in(1,2)}{2^{2+\alpha}M }\bigg(1+\frac{M}{\alpha(2-\alpha)(\alpha-1)}\bigg)+\mathbf{1}_{\alpha=1}\Big( {4M^2}\big(e^{2+1/e}+ \frac{37}{9}\big)+{4M }\Big),\\
{\rm G}_{2}&=\frac{8M^{2}}{\alpha(2-\alpha)}+M^2 {\rm E}_{1}\mathbf{1}_{\alpha\in(1,2)}.\nonumber
\end{align}

 \end{proof}

\subsection{Proof of Theorem \ref{lambda2}}

We first introduce two auxiliary Lemmas whose proof can be found in the appendix.
\begin{lemma}\label{psalpha}
Let $\nu$ be a symmetric L\'evy measure with density $f$ with respect to the Lebesgue measure and $f\in \mathscr L_{M,\alpha}$ for some $\alpha\in[1,2)$ and $M>0$. Let $\eps\in(0,1]$, there exist three positive constants ${\rm K}_1$, ${\rm K}_2$ and ${\rm K}_3$, only dependent on $\alpha$, such that for all $0<t\leq \frac{(2-\alpha)\eps^\alpha}{2^{1+\alpha}M}$, it holds:
\begin{align*}
\frac{\PP(|M_t(3\eps/4)|>\eps)}{2}\leq \frac{M^2t^2{\rm K}_1}{\eps^{2\alpha}}+ t^2  \eps^{-2\alpha}M^2 {\rm K}_2\1_{\alpha\in(1,2)}+\frac{t^4 M^4 {\rm K}_3}{\eps^{4\alpha}}+\frac{32 M^2 t^2}{\eps^2} \ln(2)\1_{\alpha=1}.
\end{align*}
For explicit formulas for ${\rm K}_1,\ {\rm K}_{2}$ and ${\rm K}_3$ see \eqref{eq:k1alpha} and \eqref{k1alpha}.
\end{lemma}
\begin{lemma}\label{yalpha}
Let $\nu$ be a symmetric L\'evy measure with density $f$ with respect to the Lebesgue measure and $f\in \mathscr L_{M,\alpha}$ for some $\alpha\in[1,2)$ and $M>0$. Let $\eps>0$, set $\rho=3/4(\eps\wedge1)$ and assume that $f$ is $M(\eps\wedge 1)^{-(2+\alpha)}$-Lipschitz on the interval $((3/4(\eps\wedge1),2\eps-3/4(\eps\wedge1))$. Then, for all $t>0$ it holds: 
\begin{align*}
\big|\lambda_\rho t \PP(|M_t(\rho)+Y_1^{(\rho)}|>\eps)-\lambda_\eps t \big|&\leq   M^2 t^2\big({\rm K}_4 \eps^{-2\alpha} \1_{0<\eps\leq 1}+ \eps^{2}{\rm K}_5 \1_{\eps>1}\big)  +{\rm K}_6M t^2\lambda_{1}(\eps\wedge1)^{-\alpha},
\end{align*}
where ${\rm K}_4,\ {\rm K}_5$ and ${\rm K}_{6}$ are positive universal constants, only depending on $\alpha$, defined  in \eqref{c1c6}.
 \end{lemma}
\begin{proof}[Proof of Theorem \ref{lambda2}]
Let $\rho:=3/4(\eps\wedge1)$, 
 using \eqref{eq:decprf} at point $\rho$ and  $\PP(|M_t(\rho)|>\eps)\le\PP(|M_t(\rho)|>1\wedge\eps)$, we derive 
\begin{align*}
|\PP(|X_t|>\eps)-\lambda_\eps t|&\leq \PP(|M_t(\rho)|>1\wedge\eps)+|\lambda_\rho t \PP(|M_t(\rho)+Y_1^{(\rho)}|>\eps)-\lambda_\eps t |\\
&\quad + \lambda_\rho t (1-e^{-\lambda_\rho t}) + \PP(N_t(\rho)\geq 2)\\
:&=I_{1}+I_{2}+I_{3}+I_{4}.
\end{align*}

By  Lemma \ref{psalpha}, Lemma \ref{yalpha}, \eqref{poisson} and \eqref{eq:prflbd} it follows that
\begin{align*}
I_1+I_2+I_3+I_4&\leq  t^2M^2 \big( ({\rm F}_{1}\eps^{-2\alpha}+\lambda_1\eps^{-\alpha}{\rm F}_2) \1_{0<\eps\leq 1}+(\eps^2{\rm F}_3 +{\rm F}_{4})\1_{\eps>1} \big) \nonumber\\
&\quad +t^4 M^4 {\rm F}_5(\eps\wedge 1)^{-4\alpha}+2t^2\lambda_1^2,
\end{align*}
where ${\rm F}_{2}={\rm K}_{6} $, ${\rm F}_{3}={\rm K}_{5}$,  ${\rm F}_{5}=2{\rm K}_{3} $ and
\begin{align}\label{d1d2}
{\rm F}_{1}&:=2{\rm K}_1+2{\rm K}_2\1_{\alpha\in(1,2)}+{\rm K}_4+64\ln(2)\1_{\alpha=1}+\frac{6}{\alpha^2},\nonumber\\ 
{\rm F}_{4}&:=2{\rm K}_1+2{\rm K}_2\1_{\alpha\in(1,2)}+64\ln(2)\1_{\alpha=1}+\frac{6}{\alpha^2}.
\end{align}
\end{proof}

\appendix
\section{Technical lemmas and additional proofs\label{secApp}}

\subsection{Proof of Lemma \ref{sj}}
For any $u>0$ we have that 
\begin{equation*}\E\big[e^{u M_t(\eps)}\big]\leq \exp\bigg(t\int (e^{u|y|}-u|y|-1)\nu_\eps(dy) \bigg)
\end{equation*}
and therefore, using that $\int |y|^k\nu_\eps(dy)\leq \eps^{k-2}\sigma^2(\eps)$ for all $k\geq 2$,
\begin{align}\label{eq:lem}
\PP(M_t(\eps)>x)&\leq \exp\Big(-ux+t\int (e^{u|y|}-u|y|-1)\nu_\eps(dy)\Big)\nonumber \\
&= e^{\frac{u^2 t\sigma^2(\eps)}{2}-ux+t \sum_{k=3}^\infty \int \frac{u^k|y|^k}{k!}\nu_\eps(dy)}\leq e^{\frac{u^2 t\sigma^2(\eps)}{2}-ux +t \sigma^2(\eps)\sum_{k=3}^\infty  \frac{u^k\eps^{k-2}}{k!}}\nonumber \\
&=e^{-ux +t \frac{\sigma^2(\eps)}{\eps^2}(e^{u\eps}-1-u\eps)}.
\end{align}
Injecting $u^*=\frac{1}{\eps}\log\big(1+\frac{x\eps}{t\sigma^2(\eps)}\big)$ in \eqref{eq:lem}, we find that $\PP(M_t(\eps)>x)\leq e^{\frac{x}{\eps}}\big(\frac{t\sigma^2(\eps)}{x\eps+t\sigma^2(\eps)}\big)^{\frac{x\eps+t\sigma^2(\eps)}{\eps^2}}$, as claimed.
To derive \eqref{sg}, we simply use the fact that $u^{-u}\leq e^{e^{-1}}$ for all $u> 0$. Indeed, set $u=\frac{x\eps+t\sigma^2(\eps)}{\eps^2}$ and notice that 
\begin{align*}
\Big(\frac{t\sigma^2(\eps)}{x\eps+t\sigma^2(\eps)}\Big)^{\frac{x\eps+t\sigma^2(\eps)}{\eps^2}}&=\Big(\frac{t\sigma^2(\eps)}{\eps^2}\Big)^u u^{-u}
\leq e^{e^{-1}} \Big(\frac{t\sigma^2(\eps)}{\eps^2}\Big)^{\frac{x\eps+t\sigma^2(\eps)}{\eps^2}}.
\end{align*}
Equation \eqref{sg} then follows under the assumption $t\sigma^2(\eps)\eps^{-2}\leq 1$.

Analogous arguments, with $M_t(\eps)$ replaced by $-M_t(\eps)$, allows to deduce that 
$$\PP(-M_t(\eps)>x)\leq e^{\frac{x}{\eps}}\big(\frac{t\sigma^2(\eps)}{x\eps+t\sigma^2(\eps)}\big)^{\frac{x\eps+t\sigma^2(\eps)}{\eps^2}}$$
and hence the inequality
$$\PP(M_t(\eps)\leq -x)\leq e^{e^{-1}} \Big(\frac{t\sigma^2(\eps)}{\eps^2}\Big)^{\frac{x\eps+t\sigma^2(\eps)}{\eps^2}},$$
whenever $t\sigma^2(\eps)\eps^{-2}\leq 1$.

\subsection{Proof of Lemma \ref{aux}}

First, we consider the general case where $\nu$ is not symmetric.
We control the quantity $J=\lambda_\rho\PP(|M_t(\rho)+tb(\rho)+Y_1^{(\rho)}|>\eps)-\lambda_\eps$ as $Q=|J|t$. It holds that
\begin{align}
J&=\int_\rho^\infty \big(\PP(M_t(\rho)+tb(\rho)<-\eps-z) f(z)+\PP(M_t(\rho)+tb(\rho)>\eps+z) f(-z)\big)dz\nonumber\\
&\quad -\int_\eps^\infty \big(\PP(M_t(\rho)+tb(\rho)\leq \eps-z)f(z)+\PP(M_t(\rho)+tb(\rho)>z-\eps)f(-z)\big)dz\nonumber\\
&\quad+\int_{\rho}^\eps  \big(\PP(M_t(\rho)+tb(\rho)> \eps-z)f(z)+\PP(M_t(\rho)+tb(\rho)<-\eps+z)f(-z)\big)dz\nonumber\\&=:R-S+T. \label{dec}
\end{align}
Recall $\rho=\eps\wedge1$, assumptions on $t$ ensures that $t|b(\rho)|\leq \rho/2$, thus 
\begin{align*}
R\leq \int_\rho^\infty  \big(\PP(M_t(\rho)<-\rho) f(z)+\PP(M_t(\rho)>\rho) f(-z)\big)dz.
\end{align*}
By means of Markov inequality and \eqref{eq:prfsigma} we then derive 
\begin{equation}\label{n56}
|R|\leq \frac{2M }{2-\alpha}t \lambda_\rho \rho^{-\alpha}.
\end{equation}
To treat the terms $S$ and $T$ we distinguish the cases $\eps\in(0,1]$ and $\eps>1$. Moreover, we restrict to the case $b(\rho)\geq 0$, the case $b(\rho)<0$ can be obtained similarly and leads to the same result. Decompose $S:=S_{1}+S_{1}$ where 
\begin{align*}
S_{1}+S_{2}=\int_{-tb(\rho)}^\infty \PP(M_t(\rho)> x)f(-\eps-tb(\rho)-x)dx+ \int_{tb(\rho)}^{\infty} \PP(M_t(\rho)\leq  -x)f(x+\eps-tb(\rho))dx.
\end{align*} 
We only detail the computations for the term $S_1$, those for the term $S_2$ being analogous.

\paragraph{Case $\eps\in(0,1]:$} Then $\rho=\eps$ and by means of the triangle inequality it holds 
\begin{align*}
|S_1|&\leq\int_{-tb(\eps)}^{tb(\eps)}f(-\eps-tb(\eps)-x)dx+\int_{t b(\eps)}^{\eps/2}\PP( M_t(\eps)> x)f(-\eps-tb(\eps)-x)dx\\
&\quad +\int_{\eps/2}^\eps\PP( M_t(\eps)> x)f(-\eps-tb(\eps)-x)dx+\int_{\eps}^\infty\PP( M_t(\eps)> x)f(-\eps-tb(\eps)-x)dx\\
&=S_{1,1}+S_{1,2}+S_{1,3}+S_{1,4}.
\end{align*}
Using that $f\in\mathscr L_{M,\alpha} $ and \eqref{eq:prfb}, it follows that 
\begin{align}\label{n11}
S_{1,1}&\leq 2Mtb(\eps) \eps^{-(1+\alpha)}\leq \frac{4M^2t\eps^{-2\alpha}}{1-\alpha}.
\end{align}
To control the term $S_{1,2}$ we proceed as for the control of the term $T_{2,1}$ in the proof of Theorem \ref{ps1}. Observe that $0<t\leq (1-\alpha) M^{-1} \eps^{\alpha}4^{-(1+\alpha)}$ implies $2Mt\eps^{1-\alpha} (1-\alpha)^{-1}<\eps/2$. Let $x\in(2Mt\eps^{1-\alpha} (1-\alpha)^{-1},\eps/2)$ and set $\tilde x:=x-2Mt\eps^{1-\alpha} (1-\alpha)^{-1}$. In particular we can write $M_t(\eps)=M_t(\tilde x)+Z_t(\tilde x, \eps)-t(b(\eps)-b(\tilde x))$. From the assumption $f\in\mathscr L_{M,\alpha}$ it also follows that $|b(\eps)-b(\tilde x)|\leq 2M \eps^{1-\alpha}(1-\alpha)^{-1}$ and so $\PP(M_t(\tilde x) >x+t(b(\eps)-b(\tilde x)))\leq \PP(M_t(\tilde x)>\tilde x)$. Therefore, for all $x\in(2Mt\eps^{1-\alpha} (1-\alpha)^{-1},\eps/2)$, the Markov inequality and \eqref{eq:prfsigma}, lead  to
\begin{align*}
\PP(M_t(\eps)>x)&\leq \PP(M_t(\tilde x) >x+t(b(\eps)-b(\tilde x)))+\PP(N_t(\tilde x,\eps)\geq 1)\\
&\leq \PP(M_t(\tilde x)>\tilde x)+t \lambda_{\tilde x,\eps}\leq  \frac{2(2+\alpha)Mt \tilde x^{-\alpha}}{\alpha(2-\alpha)}.
\end{align*}
Furthermore, by means of \eqref{eq:prfb}, $t|b(\eps)|\leq2Mt\eps^{1-\alpha} (1-\alpha)^{-1}$ and $f\in\mathscr L_{M,\alpha}$ we get
$$\int_{tb(\eps)}^{2M/(1-\alpha)t\eps^{1-\alpha} }\hspace{-0.3cm}f(-\eps-tb(\eps)-x)dx\leq \int_{0}^{2Mt\eps^{1-\alpha} (1-\alpha)^{-1}}f(-\eps-tb(\eps)-x)dx \leq \frac{2M^2}{1-\alpha}t\eps^{-2\alpha}$$
and
\begin{align*}
\int_{ 2Mt\eps^{1-\alpha} /(1-\alpha)}^{\eps/2}\hspace{-0.3cm}\tilde x^{-\alpha}f(-\eps-tb(\eps)-x)dx&\le
\frac{ M}{\eps^{1+\alpha}}\int_{2Mt\eps^{1-\alpha} (1-\alpha)^{-1}}^{\eps/2} \Big(x-2Mt\eps^{1-\alpha} (1-\alpha)^{-1}\Big)^{-\alpha}dx\\
 &\le  \frac{M}{1-\alpha}\eps^{-2\alpha}.
\end{align*}
We derive that
\begin{align}\label{s12}
S_{1,2}\leq\frac{2M^2}{1-\alpha}t\eps^{-2\alpha}+ \frac{2(2+\alpha)M^{2}}{\alpha(2-\alpha)(1-\alpha)}t \eps^{-2\alpha}.
\end{align}

To treat the term $S_{1,3}$ we notice that for any $t\in(0,(1-\alpha) M^{-1} \eps^{\alpha}4^{-(1+\alpha)})$ and $x\in[\eps/2,\eps]$ we have that $t\lambda_{x,\eps}\leq 1$ and hence, by Lemma \ref{cCPP}, we derive that for all $x\in[\eps/2,\eps]$ it holds:
\begin{align*}
\PP(M_t(\eps)>x)\leq \PP(M_t(x) +\tilde Z_t(x,\eps)>x) +2t\eps \sup_{|y|\in[x,\eps]} f(y),
\end{align*}
where $\tilde Z_t(x,\eps):=\sum_{i=1}^{N_t(x,\eta)}\big(Y_i^{(x,\eta)}-\E[Y_i^{(x,\eta)}]\big)$. 
Then, using \eqref{poisson}, the Markov inequality, \eqref{eq:prfsigma} and \eqref{eq:prflbd} we get 
\begin{align*}
 \PP(M_t(x) +\tilde Z_t(x,\eps)>x)\leq  \PP(M_t(x)>x)+\PP(N_{t}(x,\eps)\geq1)\leq \frac{2(2+\alpha)Mtx^{-\alpha}}{(2-\alpha)\alpha}.
\end{align*}
Moreover, the fact that $f\in\mathscr L_{M,\alpha}$ implies $ \sup_{|y|\in[x,\eps]} f(y)\leq M x^{-(1+\alpha)}\leq M 2^{1+\alpha} \eps^{-(1+\alpha)}$ for $x\in[\eps/2,\eps]$ and so we deduce that
\begin{align*}
\PP(M_t(\eps)>x)\leq2^{\alpha+1} Mt\eps^{-\alpha}\Big(2+\frac{2+\alpha}{(2-\alpha)\alpha}\Big).
\end{align*}
Together with $ \int_{\eps/2}^{\eps} f(-\eps-tb(\eps)-x)dx\leq \lambda_\eps$,
 we obtain
\begin{align}\label{s13}
S_{1,3}\leq\lambda_\eps2^{\alpha+1} Mt\eps^{-\alpha}\Big(2+\frac{2+\alpha}{(2-\alpha)\alpha}\Big).
\end{align}
Finally, for the term $S_{1,4}$ we have that
\begin{align*}
S_{1,4}\leq \PP(M_t(\eps)>\eps)\int_{-\infty}^{-2\eps-tb(\eps)} f(x)dx\leq  \PP(M_t(\eps)>\eps) \lambda_\eps.
\end{align*}
From the Markov inequality and \eqref{eq:prfsigma} we then derive
\begin{equation}\label{s14}
S_{1,4}\leq \frac{2M }{2-\alpha}t \lambda_\eps \eps^{-\alpha}.
\end{equation}

 Combining \eqref{n11}, \eqref{s12}, \eqref{s13} and \eqref{s14} yield
 \begin{align*}
 |S_1|&\leq\frac{ 2M^2}{1-\alpha}t\eps^{-2\alpha}\bigg(3+\frac{2+\alpha}{\alpha(2-\alpha)}\bigg)+ 2Mt\eps^{-\alpha}\lambda_\eps\bigg(1+2^\alpha\Big(2+\frac{2+\alpha}{\alpha(2-\alpha)}\Big)\bigg). 
 \end{align*}
The term $S_2$ can be controlled in a similar way, in particular it holds that 
\begin{align}\label{s1s2}
|-S|\leq \frac{ 4M^2}{1-\alpha}t\eps^{-2\alpha}\bigg(3+\frac{2+\alpha}{\alpha(2-\alpha)}\bigg)+ 4Mt\eps^{-\alpha}\lambda_\eps\bigg(1+2^\alpha\Big(2+\frac{2+\alpha}{\alpha(2-\alpha)}\Big)\bigg). 
  \end{align}
Finally, we observe that when $\eps\in(0,1]$ the term $T$ is identically zero.

Gathering Equations \eqref{dec},  \eqref{n56} and \eqref{s1s2}, we  conclude that for $\eps\in(0,1]$ 
 \begin{align*}
 |\lambda_\eps t \PP(|M_t(\eps)+tb(\eps)+Y_1^{(\eps)}|>\eps)-\lambda_\eps t |&\leq  t^{2}(M^2{\rm D}_1\eps^{-2\alpha}+M\lambda_\eps\eps^{-\alpha} {\rm D}_2),
 \end{align*}
 where
 \begin{align}\label{cte3}
 {\rm D}_1&:=\frac{4}{1-\alpha}\bigg(3+\frac{2+\alpha}{\alpha(2-\alpha)}\bigg)\quad \text{ and }\quad {\rm D}_2:=4\bigg(\frac{1}{2-\alpha}+1+2^\alpha\Big(2+\frac{2+\alpha}{\alpha(2-\alpha)}\Big)\bigg),
 \end{align}
 as claimed.
\paragraph{Case $\eps>1:$} Then $\rho=1$,  using that $f\in\mathscr L_{M,\alpha}\cap\mathscr L_M$ we readily derive 
\begin{align}\label{s1eps}
S_1\leq 2Mtb(1)+M\Big(\int_{tb(1)}^{1/2} 
+\int_{1/2}^\infty \PP(M_t(1)>x)dx\Big)=:\tilde S_{1,1}+\tilde S_{1,2}+\tilde S_{1,3}.
\end{align}
The term $\tilde S_{1,2}$ is the analogous of $S_{1,2}$ above. Observe that under the assumptions $0<t\leq (1-\alpha) (5M)^{-1}$ and $f\in\mathscr L_{M,\alpha}$, we get $t|b(1)|\leq 1/2$.
For any $x\in (2Mt(1-\alpha)^{-1},1/2)$, set $\hat x:= x-2Mt(1-\alpha)^{-1}.$ The same reasoning as for the term $S_{1,2}$ allows to conclude that, for any $x\in (2Mt(1-\alpha)^{-1},1/2)$, 
\begin{align}\label{m1bs}
\PP(M_t(1)>x)\leq \PP(M_t(\hat x)>\hat x)+t(\lambda_{\hat x,2}+\lambda_2)\leq \frac{4}{(2-\alpha)\alpha}Mt \hat x^{-\alpha} +t\lambda_2,
\end{align}
where in the last inequality we used the fact that $f\in\mathscr L_{M,\alpha}$ joined with the Markov inequality, \eqref{eq:prfsigma} and \eqref{eq:prflbd}.
Therefore, from \eqref{m1bs} and using again that $f\in\mathscr L_M$, we get
\begin{align}\label{ts12}
\tilde S_{1,2}&\leq Mt\Big(\frac{2M}{1-\alpha}-b(1)\Big)+\frac{4M^2t}{\alpha(2-\alpha)}\int_{2Mt(1-\alpha)^{-1}}^{1/2} \Big(x-\frac{2Mt}{1-\alpha}\Big)^{-\alpha}dx+\frac{tM\lambda_2}{2}\nonumber\\
&\leq Mt\Big(\frac{2M}{1-\alpha}-b(1)\Big)+\frac{4 M^2t}{\alpha(2-\alpha)(1-\alpha)}+\frac{tM\lambda_2}{2}.
\end{align}
Furthermore, by the Markov inequality and \eqref{eq:prfsigma}, we deduce that
\begin{align}\label{ts13}
\tilde S_{1,3}\leq \frac{4M^2t}{2-\alpha}.
\end{align}
Gathering \eqref{s1eps}, \eqref{ts12} and \eqref{ts13} we conclude that
\begin{align*}
S_1&\leq 2Mtb(1)+Mt\Big(\frac{2M}{1-\alpha}-b(1)\Big)+\frac{8 M^2t}{\alpha(2-\alpha)(1-\alpha)}+\frac{tM\lambda_2}{2}.
\end{align*}
Thus the term $S$ in \eqref{dec} can be bounded by
\begin{align}\label{sgrande}
|S|\leq 4Mtb(1)+2Mt\Big(\frac{2M}{1-\alpha}-b(1)\Big)+\frac{16 M^2t}{\alpha(2-\alpha)(1-\alpha)}+tM\lambda_2.
\end{align}
By means of \eqref{n56}, the term $R$ in \eqref{dec} is bounded by
\begin{align}\label{rgrande}
|R|\leq \frac{2Mt\lambda_1}{2-\alpha}.
\end{align}
To control $J$ we are  left to control the term $T$ in \eqref{dec}. We provide an upper bound for 
\begin{equation*}
T_1:=\int_1^\eps \PP(M_t(1)+tb(1)>\eps-z)f(z)dz=\int_{-tb(1)}^{\eps-1-tb(1)}\PP(M_{t}(1)\ge x)f(\eps-x-tb(1))dx,
\end{equation*}
the control of the quantity $\int_1^\eps\PP(M_t(1)+tb(1)<-\eps+z)f(-z)\big)dz$ can be treated similarly.
We have, using $t|b(1)|\le 1/2$,
\begin{align*}
T_{1}&=\1_{\eps\ge 1+2tb(1)}\bigg(\int_{-tb(1)}^{tb(1)}+\int_{tb(1)}^{1/2\wedge(\eps-1-tb(1))}\hspace{-0.2cm}+\int_{1/2\wedge(\eps-1-tb(1))}^{\eps-1-tb(1))}\PP(M_{t}(1)\ge x)f(\eps-x-tb(1))dx\bigg)\\
&\quad + \1_{1<\eps<1+2tb(1)}\int_{-tb(1)}^{tb(1)}\PP(M_{t}(1)\ge x)f(\eps-x-tb(1))dx =T_{1,1}+T_{1,2}.
\end{align*} 
For $f\in\mathscr L_{M}$, recalling the definition of $\tilde S_{1,2}$ given in \eqref{s1eps} and that we assumed $b(1)\geq0$, for $\eps\geq 1+2tb(1)$ we  write
\begin{align*}
T_{1,1}&\leq 2Mtb(1)+ M\int_{tb(1)}^{1/2\wedge (\eps-1-tb(1))}\PP(M_t(1)>y)dy+M\int_{1/2\wedge (\eps-1-tb(1))}^{\eps-1-tb(1)}\PP(M_t(1)>y)dy\\
&\leq M\Big(2tb(1)+\tilde S_{1,2}+\PP(M_t(1)>1/2)(\eps-3/2-tb(1))\1_{1/2<\eps-1-tb(1)}\Big)\\
&\leq Mt\Big(2b(1)+\frac{2M}{1-\alpha}+\frac{4 M}{\alpha(2-\alpha)(1-\alpha)}+\frac{\lambda_2}{2}+\frac{8M}{2-\alpha}(\eps-3/2-tb(1))\1_{\eps>3/2+tb(1)}\Big),
\end{align*}
where we used \eqref{ts12}, the Markov inequality and \eqref{eq:prfsigma}.
Concerning the term $T_{1,2}$, using that $f\in\mathscr L_{M,\alpha}$ joined with \eqref{eq:prfb} and the assumption $t<(1-\alpha)(5M)^{-1}$, we get
$$T_{1,2}\leq 2Mtb(1)(\eps-2tb(1))^{-1-\alpha}\1_{1<\eps<1+2tb(1)}\leq 4Mt5^{\alpha}\1_{1<\eps<1+2tb(1)}.$$

This entails that
\begin{align}\label{tgrande}
T&\leq Mt\bigg(\1_{\eps\ge 1+2t|b(1)|}\Big(\frac{6M}{1-\alpha}+\frac{4 M}{\alpha(2-\alpha)(1-\alpha)}+\frac{\lambda_2}{2}\nonumber \\
&\quad\phantom{Mt}+\frac{8M}{2-\alpha}(\eps-3/2-t|b(1)|)\1_{3/2+t|b(1)|<\eps}\Big)+4\times 5^{\alpha}\1_{1<\eps<1+2t|b(1)|}\bigg).
\end{align}
Combining \eqref{dec}, \eqref{sgrande}, \eqref{rgrande}, \eqref{tgrande} and using \eqref{eq:prflbd}, we conclude that for any $\eps>1$, $0<t< (1-\alpha)(5M)^{-1}$ and $f\in\mathscr L_{M,\alpha}\cap\mathscr L_M$ it holds, using $t|b(1)|\le 1/2$,
\begin{align*}
J&\leq 2M^2t \Big(\tilde{\rm D}_1+\frac{4}{2-\alpha}(\eps-3/2-t|b(1)|)\1_{3/2+t|b(1)|<\eps}\Big)\\
&\quad +Mt\bigg(4\times 5^{\alpha}\1_{1<\eps<1+2t|b(1)|}+\frac{8}{5}+3\lambda_2+\frac{4\lambda_1}{2-\alpha}\bigg),
\end{align*}
where we used the notation
\begin{equation}\label{tdlm3}
\tilde{\rm D}_1:=\frac{5}{1-\alpha}+\frac{10}{\alpha(2-\alpha)(1-\alpha)}.
\end{equation}
 \paragraph{Case $\nu$ symmetric and $\eps>0:$}
 In the case where $\nu$ is  symmetric  the proof can be  simplified. Since $b(\rho)\equiv0$, $M_t(\rho)=M_t(x)+Z_t(x,\rho)$ for all $x\in(0,\rho)$, $t>0$ and it holds
 \begin{align*}
 \lambda_\rho &\PP(|M_t(\rho)+Y_1^{(\rho)}|>\eps)-\lambda_\eps =2\bigg(\int_\rho^\infty \big(\PP(M_t(\rho)>\eps+z)-\PP(M_t(\rho)<\eps-z)\big)f(z)dz\bigg)\\
 &\leq \lambda_\rho \PP(M_t(\rho)>2\rho)+2\bigg(\int_0^\rho \PP(M_t(\rho)>x) f(x+\eps)dx+\PP(M_t(\rho)>\rho)\int_{\rho}^\infty  f(x+\eps)dx\bigg).
 \end{align*}
 With the same arguments as those used to treat the term $S_{1,2}$ above, one finds that for any $x\in(0,\eps)$ and $t>0$ it holds $\PP(M_t(\rho)>x)\leq \frac{2(2+\alpha)M tx^{-\alpha}}{\alpha(2-\alpha)}$. Therefore, by the Markov inequality,  \eqref{eq:prfsigma} and using that $f\in\mathscr L_{M,\alpha}$, we conclude for all $\eps\in(0,1)$ and $t>0$ it holds
 $$| \lambda_\rho \PP(|M_t(\rho)+Y_1^{(\rho)}|>\eps)-\lambda_\eps|\leq \frac{tM}{2(2-\alpha)}\big(\lambda_\eps \eps^{-\alpha}+4\lambda_{2\eps} \eps^{-\alpha}\big)+2tM^2{\rm D_2}\eps^{-2\alpha},$$
 with
 \begin{equation}\label{d2lm3}
{ \rm D}_3:=\frac{2(2+\alpha)}{(2-\alpha)\alpha(1-\alpha)}.
 \end{equation}
 If instead $\eps>1$, assuming in addition that $f\in \mathscr L_M$, we derive 
  $$| \lambda_1 \PP(|M_t(1)+Y_1^{(1)}|>\eps)-\lambda_\eps|\leq \frac{tM}{2-\alpha}\bigg(\lambda_1 2^{-\alpha}+\frac{4M}{\alpha(1-\alpha)}+\lambda_{1+\eps} \bigg).$$
 This concludes the proof.

\subsection{Proof of Lemma \ref{yalphabis}}
 Decomposition \eqref{eq:auxalpha} as in the proof of Lemma  \ref{yalpha} in $\lambda_\rho t \PP(|M_t(\rho)+Y_1^{(\rho)}|>\eps)-\lambda_\eps t=:2t (R_1+R_2)$ still holds
with 
\begin{align*}
|R_2|\leq tM \1_{0<\eps\leq 1}\bigg(\frac{M\eps^{-2\alpha}}{\alpha(2-\alpha)}+\frac{\eps^{-\alpha}\lambda_1}{2(2-\alpha)}\bigg)+tM\frac{\1_{\eps>1}\lambda_1}{(2-\alpha)(\eps+1)^2}.
\end{align*}
Set $C=\big(1\wedge\big((2-\alpha)/2M\big)\big)^{1/\alpha}$ and note that $C(\eps\wedge 1)/2>t^{1/\alpha}$. Using the symmetry of $f$ we get 
\begin{align*}
|R_1|&\leq \int_0^{t^{1/\alpha}/C}\big( f(y+\eps)+\1_{\eps>1}f(\eps-y)\big)dy +\int_{t^{1/\alpha}/C}^{\rho}\big(\PP(M_t(y)>y)+t\lambda_{y,\rho}\big)f(\eps+y)dy\\
&\quad+\PP(M_t(\rho)>\rho)\int_\rho^\infty f(y+\eps)dy\\
&\quad +\1_{\eps>1}\bigg(\int_{t^{1/\alpha}/C}^{1\wedge(\eps-1)}\big(\PP(M_t(y)>y)+t\lambda_{y,1}\big)f(\eps-y)dy+\PP(M_t(1)>1)\int_{1\wedge(\eps-1)}^{\eps-1} f(\eps-y)dy\bigg).
\end{align*}
 Next as $f\in\mathscr L_{M,\alpha}\cap \mathscr L_M$, it follows from Equations \eqref{eq:decprf}, \eqref{poisson}, \eqref{eq:prfsigma}, \eqref{eq:prflbd} and the Markov inequality, that \begin{align*}
|R_1|&\leq \frac{M t^{1/\alpha}}{C}\bigg(\frac{1}{(\eps\wedge 1)^{1+\alpha}}+\frac{\1_{\eps>1}}{(\eps-t^{1/\alpha}/C)\wedge 1}\bigg)\\
&\quad +\frac{2Mt}{2-\alpha}\Big( \rho^{-\alpha}\Big(M\alpha^{-1}(\rho+\eps)^{-\alpha}+\frac{\lambda_1}{2}\Big)+M\1_{\eps>2}(\eps-2)\Big)\\
&\quad +\frac{4M^2 t^{1/\alpha}C^{\alpha-1}\1_{\alpha\in(1,2)}}{\alpha(2-\alpha)(\alpha-1)}\bigg(\frac{1}{(\eps+t^{1/\alpha}/C)\wedge 1}+\frac{\1_{\eps>1}}{(\eps-t^{1/\alpha}/C)\wedge 1}\bigg)\\
&\quad +4M^2t\1_{\alpha=1}\bigg(\ln\Big(\frac{C(1\wedge|\eps-1|)}{t}\Big)\frac{\1_{\eps>1}}{(\eps-t/C)\wedge 1}+\ln\Big(\frac{C\rho}{t}\Big)\frac{1}{(\eps+t/C)\wedge 1}\bigg)\\
&\leq \frac{M t^{1/\alpha}}{C}\bigg(\frac{1}{(\eps\wedge 1)^{1+\alpha}}+\frac{2}{\eps\wedge 1}\bigg)+\frac{2Mt}{2-\alpha}\Big(M\alpha^{-1}(\eps\wedge1)^{-2\alpha}+\frac{\lambda_1(\eps\wedge 1)^{-\alpha}}{2}+M\1_{\eps>2}\eps\Big)\\
&\quad +\frac{12M^2 t^{1/\alpha}C^{\alpha-1}\1_{\alpha\in(1,2)}}{\alpha(2-\alpha)(\alpha-1)}\frac{1}{\eps\wedge 1} +12M^2t\1_{\alpha=1}\ln\Big(\frac{C(1\wedge\eps\wedge(\eps-1\vee 0))}{t}\Big)\frac{1}{\eps\wedge 1},
\end{align*}
with the convention that $0\ln 0=0$. Therefore,
\begin{align*}
 |\lambda_\rho t \PP(|M_t(\rho)&+Y_1^{(\rho)}|>\eps)-\lambda_\eps t|
\\&\leq{\rm L}_{1}\frac{t^{1+1/\alpha}}{(\eps\wedge 1)^{1+\alpha}}+\frac{8M^{2}}{\alpha(2-\alpha)}\frac{t^{2}}{(\eps\wedge 1)^{2\alpha}}+\frac{5M}{2-\alpha}\frac{t^{2}\lambda_{1}}{(\eps\wedge1)^2}+\frac{4M^{2}t^{2}}{2-\alpha}\1_{\eps>2}\eps\\
&\quad +12M^2t^{2}\1_{\alpha=1}\ln\Big(\frac{C(1\wedge\eps\wedge(\eps-1\vee 0))}{t}\Big)\frac{1}{\eps\wedge 1} \end{align*}
where 
\begin{align}\label{eq:L123}
{\rm L}_{1}&=\frac{2M}{C}+\frac{4M}{C}+\frac{24M^2 C^{\alpha-1}\1_{\alpha\in(1,2)}}{\alpha(2-\alpha)(\alpha-1)},
\end{align}
as desired.

\subsection{Proof of Lemma \ref{psalpha}}First, using the symmetry of $\nu$ it holds $\PP(|M_t(\rho)|>\eps)=2\PP(M_t(\rho)>\eps)$ where we write $\rho:=3\eps/4$. Since $\eps/2<\rho<\eps$ together with \eqref{eq:decprf} and \eqref{poisson}, we obtain
\begin{align}\label{eqdep}
\PP(M_t(\rho)>\eps)\leq \PP(M_t(\eps/2)>\eps)+t\lambda_{\eps/2,\rho}\PP(M_t(\eps/2)+Y_1^{(\eps/2,\rho)}>\eps)+(t\lambda_{\eps/2,\rho})^2.
\end{align}
Applying Lemma \ref{sj} and using \eqref{eq:prfsigma} we derive 
\begin{align}\label{primi}
\PP(M_t(\eps/2)>\eps)+(t\lambda_{\eps/2,\rho})^2\leq M^2t^2\eps^{-2\alpha}{\rm K}_1,
\end{align}
with
\begin{equation}\label{eq:k1alpha}
{\rm K}_1:=4^{1+\alpha}\bigg( \frac{e^{2+1/e}}{(2-\alpha)^2}+\frac{1}{\alpha^2}\bigg).
\end{equation}
Using again the symmetry of $\nu$ we can establish 
\begin{align}\label{t1t2}
\lambda_{\eps/2,\rho}\PP(&M_t(\eps/2)+Y_1^{(\eps/2,\rho)}>\eps)=\hspace{-0.1cm}\int_{\eps/4}^{\eps/2} \hspace{-0.3cm}\PP(M_t(\eps/2)>y) f(\eps-y)dy+\int_{\eps/2}^\rho \hspace{-0.2cm}\PP(M_t(\eps/2)>\eps+y)f(y)dy\nonumber \\
&\leq \int_{\eps/4}^{\eps/2} \PP(M_t(\eps/2)>y) f(\eps-y)dy+ \PP(M_t(\eps/2)>3/2\eps)\frac{\lambda_{\eps/2,\rho}}{2}=:T_1+T_2.
\end{align}
Applying \eqref{eq:decprf}, \eqref{poisson}, the Markov inequality and \eqref{eq:prfsigma}, for any $y\in (\eps/4,\eps/2)$ we have
\begin{align*}
\PP(M_t(\eps/2)>y)&\leq \PP(M_t(y)>y)+t\lambda_{y,\eps/2}\leq \frac{4Mt y^{-\alpha}}{\alpha(2-\alpha)}.
\end{align*}
It follows that
\begin{align}\label{psalphat1}
T_1&\leq \frac{4Mt }{\alpha(2-\alpha)} \int_{\eps/4}^{\eps/2} y^{-\alpha} f(\eps-y)dy\leq \frac{2^{3+\alpha}M^2t }{\alpha(2-\alpha)\eps^{1+\alpha}}\bigg(\frac{(\eps/4)^{1-\alpha}}{(\alpha-1)}\1_{\alpha\in(1,2)}+\ln(2)\1_{\alpha=1}\bigg).
\end{align}
Furthermore, note that Lemma  \ref{sj} applies as $t\leq \frac{(2-\alpha)\eps^\alpha}{2^{1+\alpha}M}$ implies $4t\sigma^2(\eps/2)\eps^{-2}\leq 1$, together with \eqref{eq:prfsigma}, it gives
\begin{align}\label{psalphat2}
T_2&\leq \frac{t^32^{3(1+\alpha)}e^{3+1/e} M^4 }{\alpha(2-\alpha)^3\eps^{4\alpha}}.
\end{align}
From \eqref{t1t2}, \eqref{psalphat1} and \eqref{psalphat2}, we obtain that 
\begin{align}\label{lmb}
\lambda_{\eps/2,\rho}\PP(M_t(\eps/2)+Y_1^{(\eps/2,\rho)}>\eps)\leq &t  \eps^{-2\alpha}M^2 {\rm K}_2\1_{\alpha\in(1,2)}+\frac{t^3 M^4 {\rm K}_3}{\eps^{4\alpha}}
+\frac{16 M^2 t}{\eps^2} \ln(2)\1_{\alpha=1},
\end{align}
with
\begin{align}\label{k1alpha}
 {\rm K}_2&:=\frac{2^{1+3\alpha}}{\alpha(2-\alpha)(\alpha-1)}  \quad
 \text{ and } \quad   {\rm K}_3:=\frac{2^{3(1+\alpha)}e^{3+1/e} }{\alpha(2-\alpha)^3}.
\end{align}
Finally, gathering \eqref{eqdep}, \eqref{primi} and \eqref{lmb}, we derive
$$\PP(M_t(\rho)>\eps)\leq M^2t^2\eps^{-2\alpha}{\rm K}_1+ t^2  \eps^{-2\alpha}M^2 {\rm K}_2\1_{\alpha\in(1,2)}+\frac{t^4 M^4 {\rm K}_3}{\eps^{4\alpha}}+\frac{16 M^2 t^2}{\eps^2} \ln(2)\1_{\alpha=1}.$$

\subsection{Proof of Lemma \ref{yalpha}}
First, since $\nu$ is symmetric, it holds
\begin{align*}
\lambda_\rho \PP(|M_t(\rho)+Y_1^{(\rho)}|>\eps)
&=2\int_\rho^\infty \big(\PP(M_t(\rho)>\eps-z)+\PP(M_t(\rho)>\eps+z)\big)\nu(dz).
\end{align*}
Moreover, since $\rho<\eps,$ and using again the symmetry, we obtain
\begin{align}\label{eq:auxalpha}
\nonumber \lambda_\rho t& \PP(|M_t(\rho)+Y_1^{(\rho)}|>\eps)-\lambda_\eps t \\&=2t\Big[ \int_\rho^\eps \PP(M_t(\rho)>\eps-z)\nu(dz)
-\int_\eps^\infty \PP(M_t(\rho)\leq \eps-z)\nu(dz)\Big]\nonumber \\
&\quad +2t\int^\infty_{\rho}\PP(M_t(\rho)>\eps+z)\nu(dz)
=:2t (R_1+R_2).
\end{align}
We begin by controlling the term $R_1$. Recalling that $\rho=3/4(\eps\wedge 1)$ and setting $\eta:=\eps-3/4(\eps\wedge1)$, we have:
\begin{align*}
&\int_\rho^\eps \PP(M_t(\rho)>\eps-z)\nu(dz)=\int_0^{ \eta}\PP(M_t(\rho)>x)f(\eps-x)dx,\\
&\int_\eps^\infty \PP(M_t(\rho)\leq \eps-z)\nu(dz)=\int_\eps^\infty \PP(M_t(\rho)> z-\eps)\nu(dz)=\int_0^\infty \PP(M_t(\rho)> x)f(\eps+x)dx,
\end{align*} where we used the symmetry of $\nu$ in the second line.
The triangle inequality gives
\begin{align}\label{r1ti}
|R_1|&\leq \bigg|\int_0^{ \eta}\PP(M_t(\rho)>x)(f(\eps-x)-f(\eps+x))dx\bigg|+\bigg|\int_{\eta}^\infty \PP( M_t(\rho)> x)f(\eps+x)dx\bigg|\nonumber \\
&=:R_{1,1}+R_{1,2}.
\end{align}

Therefore, by means of \eqref{eq:decprf}, \eqref{poisson}, the Markov inequality, \eqref{eq:prfsigma}, and that $f$ is $M(\eps\wedge 1)^{-(2+\alpha)}$-Lipschitz on the interval $\big(3/4(\eps\wedge 1),2\eps-3/4(\eps\wedge 1)\big)$, it follows that
\begin{align}\label{r11}
R_{1,1}&\leq 2M(\eps\wedge 1)^{-(2+\alpha)}\bigg[\1_{0<\eps\leq 1}\int_0^{\eps/4}(\PP(M_t(x)>x)+t\lambda_{x,3/4\eps})xdx\nonumber \\
& +\1_{\eps> 1}\int_0^{(\eps-3/4)\wedge 3/4}(\PP(M_t(x)>x)+t\lambda_{x,3/4})xdx+\1_{\eps> 1}\PP(M_t(3/4)>3/4)\int_{(\eps-3/4)\wedge3/4}^{\eps-3/4}xdx\bigg]\nonumber\\
\leq& \frac{8tM^2(\eps\wedge 1)^{-(2+\alpha)}}{\alpha(2-\alpha)}\bigg(\1_{0<\eps\leq 1}\int_0^{\eps/4}\frac{dx}{x^{\alpha-1}}+\1_{\eps> 1}\int_0^{(\eps-3/4)\wedge 3/4}\frac{dx}{x^{\alpha-1}}\bigg)+\1_{\eps\geq 3/2}\frac{4^{1+\alpha}3^{-\alpha}M^2t\eps^2}{2-\alpha}\nonumber\\ 
&\leq\frac{2^{2\alpha-1}}{\alpha(2-\alpha)^{2}}M^{2}t\eps^{-2\alpha}\mathbf{1}_{0<\eps\le1}+\frac{8tM^{2}}{\alpha(2-\alpha)^{2}}(\eps-3/4)^{2-\alpha}\mathbf{1}_{1<\eps\le3/2}+\frac{4^{1+\alpha}\eps^{2}M^{2}t}{3^{\alpha}(2-\alpha)}\mathbf{1}_{\eps>3/2}.
\end{align}
Concerning the term $R_{1,2}$ we have: \begin{align}\label{R12}
R_{1,2}&\leq \1_{0<\eps\leq 1}\bigg(\int_{\eps/4}^{3/4\eps}\PP(M_{t}(3/4\eps)>x)f(x+\eps)dx+\PP(M_{t}(3/4\eps)>3/4\eps)\int_{3/4\eps}^{\infty} f(x+\eps)dx\bigg) \nonumber \\
&\quad +\1_{1< \eps<3/2}\bigg(\int_{\eps-3/4}^{3/4}\PP(M_{t}(3/4)>x)f(x+\eps)dx+ \PP(M_{t}(3/4)>3/4)\int_{3/4}^{\infty}f(x+\eps)dx\bigg) \nonumber \\
&\quad +\1_{ \eps\geq 3/2} \PP(M_t(3/4)>3/4)\int_{\eps-3/4}^\infty f(x+\eps)dx.
\end{align}
Using \eqref{eq:decprf}, \eqref{poisson}, the Markov inequality, \eqref{eq:prfsigma} and \eqref{eq:prflbd}, we get
\begin{align}
\PP(M_{t}(3/4\eps)>x)&\leq \PP(M_{t}(x/2)>x)+t\lambda_{x/2,3/4\eps}\leq \frac{2^{2+\alpha}Mtx^{-\alpha}}{\alpha(2-\alpha)} ,\quad \forall x\le \frac{3\eps}{2}, \ \eps<1, \nonumber\\
\PP\big(M_{t}(3/4)&>3/4(\eps\wedge 1)\big)\leq   tM\frac{2^{2\alpha+1}}{3^{\alpha}(2-\alpha)}(\eps\wedge 1)^{-\alpha}. \label{bho}
\end{align}
Therefore, from \eqref{R12}, \eqref{bho} and \eqref{eq:prflbd} we derive using that $f\in \mathscr L_{M,\alpha}\cap \mathscr L_{M}$:
\begin{align}\label{r12bis}
R_{1,2}&\leq  \1_{0<\eps\leq 1} \bigg(tM^2\eps^{-2\alpha}\Big(\frac{2^{3\alpha}}{\alpha(\alpha-1)(2-\alpha)}+\frac{2^{4\alpha+1}}{21^{\alpha}\alpha(2-\alpha)}\Big)+tM\eps^{-\alpha}\lambda_{1}\frac{2^{2\alpha+1}}{3^{\alpha}(2-\alpha)}\bigg)\nonumber \\
&\quad+\1_{1< \eps<3/2}tM\bigg(\frac{2^{3\alpha}M}{\alpha(\alpha-1)(2-\alpha)}+\frac{2^{2\alpha+1}}{3^{\alpha}(2-\alpha)}\lambda_{7/4}\bigg)+\1_{ \eps\geq 3/2} tM\lambda_{9/4}\frac{2(4/3)^\alpha}{2-\alpha}.
\end{align}

Gathering Equations \eqref{r1ti}, \eqref{r11} and \eqref{r12bis}, we get 
\begin{align}\label{r1}
R_{1}&\leq \1_{0<\eps\leq 1} \bigg(tM^2\eps^{-2\alpha}\Big(\frac{2^{2\alpha-1}}{\alpha(2-\alpha)^{2}}+\frac{2^{3\alpha}}{\alpha(\alpha-1)(2-\alpha)}+\frac{2^{4\alpha+1}}{21^{\alpha}\alpha(2-\alpha)}\Big)+tM\eps^{-\alpha}\lambda_{1}\frac{2^{2\alpha+1}}{3^{\alpha}(2-\alpha)}\bigg)\nonumber \\
&\quad+\1_{1< \eps<3/2}tM\bigg( \frac{8M}{\alpha(2-\alpha)^{2}}(\eps-3/4)^{2-\alpha}+\frac{2^{3\alpha}M}{\alpha(\alpha-1)(2-\alpha)}+\frac{2^{2\alpha+1}}{3^{\alpha}(2-\alpha)}\lambda_{7/4}\bigg)\\ &\quad+\1_{ \eps\geq 3/2}\Big(\frac{4^{1+\alpha}\eps^{2}M^{2}t}{3^{\alpha}(2-\alpha)}+ tM\lambda_{9/4}\frac{2(4/3)^\alpha}{2-\alpha}\Big).\nonumber
\end{align}

To complete the proof we are left to control the term $R_2$ in \eqref{eq:auxalpha}. The Markov inequality, \eqref{eq:prfsigma}, the symmetry of $\nu$ and the fact that $\rho>1/2(\eps\wedge1)$ yield
\begin{align}\label{r2}
R_2&\leq \frac{\PP(M_t(\rho)>\rho)}{2}(\lambda_{\rho,1}+\lambda_1)\leq \frac{2^{\alpha}Mt(1\wedge\eps)^{-\alpha}}{2-\alpha}(2^{\alpha+1}M(1\wedge\eps)^{-\alpha}\alpha^{-1}+\lambda_1).
\end{align}
Therefore, from \eqref{eq:auxalpha}, \eqref{r1} and \eqref{r2} we conclude that
\begin{align*}
\big|\lambda_\rho t \PP(|M_t(\rho)+Y_1^{(\rho)}|>\eps)-\lambda_\eps t \big|&\leq   M^2 t^2\big({\rm K}_4 \eps^{-2\alpha} \1_{0<\eps\leq 1}+ \eps^{2}{\rm K}_5 \1_{\eps>1}\big)  +{\rm K}_6M t^2\lambda_{1}(\eps\wedge1)^{-\alpha} ,
\end{align*}
where ${\rm K}_4,\ {\rm K}_5$ and ${\rm K}_{6}$ are positive universal constants, only depending on $\alpha$, defined as follows:
\begin{align}\label{c1c6}
{\rm K}_4&:=\frac{2^{2\alpha-1}}{\alpha(2-\alpha)^{2}}+\frac{2^{3\alpha}}{\alpha(\alpha-1)(2-\alpha)}+\frac{2^{4\alpha+1}}{21^{\alpha}\alpha(2-\alpha)}+\frac{2^{2\alpha+2}}{\alpha(2-\alpha)},\nonumber\\
\quad {\rm K}_5&:=\Big( \frac{8(3/4)^{2-\alpha}}{\alpha(2-\alpha)^{2}}++\frac{2^{3\alpha}}{\alpha(\alpha-1)(2-\alpha)}\Big)\1_{1<\eps<3/2}+\frac{4^{1+\alpha}}{3^{\alpha}(2-\alpha)}\1_{\eps\ge3/2},\\
  {\rm K}_6&:= \1_{0<\eps\leq 1} \frac{2^{2\alpha+1}}{3^{\alpha}(2-\alpha)}+\1_{1< \eps<3/2}\frac{2^{2\alpha+1}}{3^{\alpha}(2-\alpha)}+\1_{ \eps\geq 3/2}\frac{2(4/3)^\alpha}{2-\alpha}. \nonumber\end{align}

\subsection{A result for compound Poisson processes}

\begin{lemma}\label{cCPP}
Let $N$ a Poisson random variable with mean $0<\lambda\leq 2$ and $(Y_i)_{i\geq 0}$ a sequence of i.i.d. random variables independent of $N$ with bounded density $g$ (with respect to the Lebesgue measure). Furthermore, let $Z$ be any random variable independent of $(N,(Y_i)_{i\geq 0})$. Then, for all $x\in\R$,
\begin{align*}
\bigg|\PP\bigg(Z+\sum_{i=1}^NY_i-\E\Big[\sum_{i=1}^NY_i\Big]>x\bigg)-\PP\bigg(Z+\sum_{i=1}^N(Y_i-\E[Y_i])>x\bigg)\bigg|\leq 2\lambda e^{-\lambda} |\E[Y_1]| \|g\|_{\infty}.
\end{align*}  
If, instead, $1<\lambda<2$, then for all $x\in\R$,
\begin{align*}
\bigg|\PP\bigg(Z+\sum_{i=1}^NY_i-\E\Big[\sum_{i=1}^NY_i\Big]>x\bigg)-\PP\bigg(Z+\sum_{i=1}^N(Y_i-\E[Y_i])>x\bigg)\bigg|\leq 2\lambda^2e^{-\lambda} |\E[Y_1]| \|g\|_{\infty}.
\end{align*}  
\end{lemma}
\begin{proof}First, note that 
\begin{align*} &\bigg|\PP\bigg(\sum_{i=1}^NY_i-\E\Big[\sum_{i=1}^NY_i\Big]>x\bigg)-\PP\bigg(\sum_{i=1}^N(Y_i-\E[Y_i])>x\bigg)\bigg|\\
&\leq\sum_{n=0}^\infty \bigg|\PP\bigg(\sum_{i=1}^nY_i>x+\lambda\E[Y_1]\bigg)-\PP\bigg(\sum_{i=1}^nY_i>x+n\E[Y_1]\bigg)\bigg|\PP(N=n)\\
&\leq \sum_{n=0}^\infty \PP(N=n)|\E[Y_1]||n-\lambda| \|g^{*n}\|_{\infty}\leq  \|g\|_{\infty}|\E[Y_1] |\sum_{n=0}^\infty \PP(N=n)|n-\lambda|.
\end{align*}
Finally we observe that, since $\lambda\leq1$, it holds
\begin{align*}
 \sum_{n=0}^\infty \PP(N=n)|n-\lambda|&=\lambda\PP(N=0)+ \E[N]-\lambda\PP(N\geq 1)=2\lambda e^{-\lambda}.
\end{align*}
If, instead, $1<\lambda<2$, then
\begin{align*}
 \sum_{n=0}^\infty \PP(N=n)|n-\lambda|&=\lambda\PP(N=0)+(\lambda-1)\PP(N=1)+\sum_{n=2}^\infty (n-\lambda) \PP(N=n)
=2\lambda^2 e^{-\lambda}.
\end{align*}
We conclude the proof by observing that for any real random variable $Z_1$ independent of $Z_2$ and $Z_3$ and any $z\in\R$ it holds 
\begin{align*}
|\PP(Z_1+Z_2>z)-\PP(Z_1+Z_3>z)|&=\bigg|\int_{\R}( \PP(Z_2>z-y)-\PP(Z_1>z-y))\mu(dy)\bigg|\\
&\leq\sup_{x\in\R} |\PP(Z_2>x)-\PP(Z_3>x)|,
\end{align*}
where $\mu$ is the law of $Z_1$.
\end{proof}

\subsection{Proofs of the Examples}\label{proof:ex}
\begin{enumerate}
\item \textbf{Compound Poisson processes.} Let $X$ be a compound Poisson process with intensity $\lambda=\nu(\R)<\infty$ and jump density $f/\lambda$. Write $X_t=\sum_{i=0}^{N_t} Z_i$, for any $\eps>0$, it holds
\begin{align*}
\PP(|X_t|>\eps)
&=t\lambda_{\eps} e^{-\lambda t}+\sum_{n=2}^\infty \PP\bigg(\Big|\sum_{i=1}^n Z_i\Big|>\eps\bigg)\PP(N_t=n).
\end{align*}
Using $\PP(N_{t}\ge 2)=O(t^{2})$ we obtain
$|\PP(|X_t|>\eps)-t\lambda_\eps|=O(t^2),$ as $t\to0.$ For $f$ a L\'evy density such that $f=f\1_{[\eps,\infty)}$, it holds $\lambda=\lambda_{\eps}$ and  later computations simplify in \[\PP(|X_{t}|>\eps)=\PP(N_{t}\ge 1)=1-e^{-\lambda_{\eps}t}=\lambda_{\eps}t-t^{2}\sum_{k\ge 2}t^{k-2}(-\lambda_{\eps})^{k}/k!.\]  In that case, the rate is exactly of the order of $t^{2}$. Next considering the small jumps, it holds for $\eps\in(0,1]$ that $tb(\eps)+M_{t}(\eps)=\sum_{i=1}^{N_{t}^{(0,\eps)}}Y_{i}^{(0,\eps)}$ and using \eqref{poisson}
\[\PP(|tb(\eps)+M_{t}(\eps)|\ge \eps)=\sum_{k=2}^{\infty}\PP(N_{t}^{(0,\eps)}=k)\PP(\big|\sum_{i=0}^{k}Y^{(0,\eps)}_{i}\big|\ge \eps)\le t^{2}(\lambda-\lambda_{\eps})^{2}.\]It is exactly oforder $t ^{2}$ for any L\'evy density such that $f=f\1_{[3\eps/4,\infty)}$.
\item \textbf{Gamma processes.}
Set $\Gamma(t,\varepsilon)=\int_\varepsilon^\infty x^{t-1}e^{-x}dx$, such that $\Gamma(t,0) = \Gamma(t)$. Using that $\Gamma(t,\varepsilon)$ is analytic we can write 
\begin{align}
 \Big|\lambda_{\eps}-\frac{\PP(X_{t}>\eps)}{t}\Big| &=\frac{1}{\Delta\Gamma(t)}\Big|\Delta\Gamma(t,0)\Gamma(0,\varepsilon)  -  \sum_{k=0}^\infty \frac{\Delta^k}{k!} \Big\{\frac{\partial^k}{\partial t^k}\Gamma(t,\varepsilon)\Big|_{t=0}\Big\}\Big| \nonumber\\
 &\leq\Gamma(0,\varepsilon)\Big|\frac{1-t\Gamma(t,0)}{t\Gamma(t)}\Big|+\Big|\frac{1}{t\Gamma(t)}\sum_{k=1}^\infty \frac{t^k}{k!} \Big\{\frac{\partial^k}{\partial t^k}\Gamma(t,\varepsilon)\Big|_{t=0}\Big\}\Big|. \label{eq:gamma}
\end{align}
As $\Gamma(t,0)$ is a meromorphic function with a simple pole in $0$ and residue 1, there exists a sequence $(a_k)_{k\geq 0}$ such that $\Gamma(t)=\frac{1}{t}+\sum_{k=0}^\infty a_kt^k$. Therefore,
\begin{equation*}
 1-t\Gamma(t,0)=t\sum_{k=0}^\infty a_kt^k,
\end{equation*}
and  
$$\frac{1-t\Gamma(t)}{t\Gamma(t)}=\frac{t\sum_{k=0}^\infty a_kt^k}{1+t\sum_{k=0}^\infty a_kt^k} = O(t),\quad \text{as }t\to 0.$$
Let us now study the term 
$\sum_{k=1}^\infty \frac{t^k}{k!} \big(\frac{\partial^k}{\partial t^k}\Gamma(t,\varepsilon)\big)\big|_{t=0}$.
We have:
\begin{align}
\Big|\frac{\partial^k}{\partial t^k}\Gamma(t,\varepsilon)\Big|_{t=0}\Big|
	&\leq  \Big|e^{-1}\int_\varepsilon^1x^{-1}(\log(x))^kdx\Big| + \Big|\int_{1}^\infty e^{-x}(\log(x))^kdx\Big| \nonumber \\
&=e^{-1}\frac{|\log(\varepsilon)|^{k+1}}{k+1}+\int_{1}^\infty e^{-x}(\log(x))^kdx. \nonumber
\end{align}
Let $x_0$ be the largest real number such that $e^{\frac{x_0}{2}}=(\log(x_0))^k$. This equation has two solutions if and only if $k\geq6$. If no such point exists, take $x_0 = 1$. Then,
\begin{align*}
 \int_{1}^\infty \hspace{-0.3cm}
 e^{-x}(\log(x))^kdx&\leq \int_1^{x_0}\hspace{-0.3cm}e^{-x}(\log(x))^kdx+\int_{x_0}^\infty \hspace{-0.3cm}e^{-\frac{x}{2}}dx\leq (\log(x_0))^k \big(e^{-1}-e^{-x_0}\big)+2e^{-\frac{x_0}{2}}\\
 &\leq e^{\frac{x_0}{2}-1} + e^{-\frac{x_0}{2}} \leq k^k+1,
\end{align*}
where we used the inequality $x_0 < 2k \log k$, for each integer $k$. Summing up, we get
\begin{align*}
 \Big|\sum_{k=1}^\infty \frac{t^k}{k!} \Big\{\frac{\partial^k}{\partial t}&\Gamma(t,\varepsilon) \Big|_{t=0} \Big\}\Big|\leq e^{-1}
\sum_{k=1}^\infty \frac{t^k}{k!}\frac{|\log(\varepsilon)|^{k+1}}{k+1}+\sum_{k=1}^5 2e^{-\frac{1}{2}}\frac{t^k}{k!} + \sum_{k=6}^\infty \frac{t^k}{k!} (k^k+1)\\
&\leq |\log(\varepsilon)|\big[e^{t|\log(\varepsilon)|}-1\big]+\sum_{k=6}^\infty \frac{t^{\frac{k}{2}}}{k!}\Big(\frac{k}{e}\Big)^k + O(t)\leq (\log(\varepsilon))^2t + O(t).
\end{align*}
In the last two steps, we have used first that $t < e^{-2}$ and then the Stirling approximation formula to deduce that the last remaining sum is $O(t^3)$. Clearly, the factor $\frac{1}{t\Gamma(t)}\sim 1$, as $t\to0$, in \eqref{eq:gamma} does not change the asymptotic. Finally we derive that 
$$
\big|t\lambda_{\varepsilon} -\PP(X_t>\varepsilon)\Big| = O\big(t^2\big),\quad \text{as }t\to 0,
$$
as desired.

\item \textbf{Inverse Gaussian processes.}
To show Equation \eqref{IGP} we write

\begin{align*}
 \bigg|\frac{\PP(X_t>\varepsilon)}{t}-\lambda_\varepsilon\bigg|\leq \bigg|e^{2t\sqrt\pi}\int_\varepsilon^\infty\frac{e^{-x}\big(e^{-\frac{\pi t^2}{x}}-1\big)}{x^{\frac{3}{2}}}dx\bigg|+\big(e^{2t\sqrt\pi}-1\big)\int_\varepsilon^\infty \frac{e^{-x}}{x^{\frac{3}{2}}}dx=: I+II.
\end{align*}
After writing the exponential $e^{-\frac{\pi t^2}{x}}$ as an infinite sum, we get $I=O(t^2)$ if $t \to 0$. Expanding $e^{2t\sqrt\pi}$ one finds that, under the same hypothesis, $II=O(t)$.

\item \textbf{Cauchy processes.}
Observe that $\lambda_\varepsilon=\frac{2}{\pi \varepsilon}$ and 
$\PP(|X_t|>\varepsilon)=\frac{2}{\pi}\big(\frac{\pi}{2}-\arctan\big(\frac{\varepsilon}{t}\big)\big)$.
Hence, in order to prove \eqref{eq:cauchy}, it is enough to show that
\begin{equation}\label{eq:cauchybis}
\lim_{t\to 0}\frac{2}{\pi}\bigg|\frac{\varepsilon^3}{t^3}\bigg(\frac{\pi}{2}-\arctan\Big(\frac{\varepsilon}{t}\Big)\bigg)-\frac{\varepsilon^2}{t^2}\bigg|<\infty. 
\end{equation}
Set $y=\frac{t}{\varepsilon}$ and we compute the limit in \eqref{eq:cauchybis} by means of de l'H\^opital rule:
\begin{align*}
 \frac{2}{\pi}\lim_{y\to 0}\bigg|\frac{1}{y^3}\bigg(\frac{\pi}{2}-\arctan\Big(\frac{1}{y}\Big)\bigg)-\frac{1}{y^2}\bigg|&=
 \frac{2}{\pi}\lim_{y\to 0}\bigg|\frac{\frac{\pi}{2}-\arctan\Big(\frac{1}{y}\Big)-y}{y^3}\bigg|\\
 &=\lim_{y\to 0}\frac{y^2}{(1+y^2)3\pi y^2}<\infty.
\end{align*}

\end{enumerate}
\paragraph{Acknowledgements.} {\small The work of E. Mariucci has been partially funded by the Federal Ministry for Education and Research through the Sponsorship provided by the Alexander von Humboldt Foundation, by
the Deutsche Forschungsgemeinschaft (DFG, German Research Foundation) – 314838170,
GRK 2297 MathCoRe, and by Deutsche Forschungsgemeinschaft (DFG) through grant CRC 1294 'Data Assimilation'.
Also, she would like to thank Francesco Caravenna, Jean Jacod and Markus Rei\ss~for several enlightening discussions which initially motivated the underlying research of this paper.}

\bibliographystyle{apalike}
\bibliography{refs}
\end{document}